\documentclass{amsart}

\usepackage{amssymb}

\usepackage{amsmath, amssymb,epic,graphicx,mathrsfs,enumerate}
\usepackage[all]{xy}

\usepackage{amsthm}
\usepackage{amssymb}
\usepackage{latexsym}
\usepackage{longtable}
\usepackage{epsfig}
\usepackage{amsmath}
\usepackage{hhline}

\usepackage{latexsym}

\newtheorem{proposition}{Proposition}[section]

\newtheorem{lemma}[proposition]{Lemma}
\newtheorem{corollary}[proposition]{Corollary}
\newtheorem{theorem}[proposition]{Theorem}
\newtheorem*{theorem1}{Theorem 1.8}
\newtheorem{example}[proposition]{Example}

\def\spl#1#2{\mathord{\mathrm{Sp}_{#1}(#2)}}
\def\psp#1#2{\mathord{\mathrm{PSp}_{#1}(#2)}}
\def\ort#1#2#3{\mathord{\mathrm{O}_{#1}^{#2}(#3)}}

\def\sort#1#2#3{\mathord{\mathrm{SO}_{#1}^{#2}(#3)}}

\newcommand{\Aut}{{\mathrm {Aut}}}

\newcommand{\Out}{{\mathrm {Out}}}

\newcommand{\F}{{\mathbb F}}

\def\Aut{\rm Aut}

\begin{document}
\title[Normalizers of Primitive Groups]{Normalizers
of Primitive Permutation Groups}
\author[R. M. Guralnick]{Robert M. Guralnick}
\address{Department of Mathematics, University of Southern California,
Los Angeles,
CA 90089-1113, USA}
\email{guralnic@math.usc.edu}

\author[A. Mar\'oti]{Attila Mar\'oti}
\address{Alfr\'ed R\'enyi Institute of Mathematics, Hungarian Academy of Sciences, Re\'altanoda utca 13-15, H-1053, Budapest, Hungary}
\email{maroti.attila@renyi.mta.hu}

\author[L. Pyber]{L\'aszl\'o Pyber}
\address{Alfr\'ed R\'enyi Institute of Mathematics, Hungarian Academy of Sciences, Re\'altanoda utca 13-15, H-1053, Budapest, Hungary}
\email{pyber.laszlo@renyi.mta.hu}

\date{January 28, 2017}
\thanks{The first author was supported by NSF grants DMS-1302886 and DMS-1600056. The second author was supported by the MTA R\'enyi ``Lend\"ulet" Groups and Graphs Research Group and by a Humboldt Return Fellowship. The second and third authors were supported by the National Research, Development and Innovation Office (NKFIH) Grant No.~K84233 and No.~K115799. 
The third author was funded by the National Research, Development and Innovation Office (NKFIH) Grant No.~ERC$\_$HU$\_$15 118286. The first and third authors would like to thank the London Mathematical Society for its support during the Durham Conference on Groups, Geometry and Combinatorics -- 2001 where some of this work was done. The authors thank Mikl\'os Ab\'ert for various eye-opening remarks.}

\begin{abstract}
Let $G$ be a transitive normal subgroup of a permutation group $A$ of finite degree $n$. The factor group $A/G$ can be considered as a certain Galois group and one would like to bound its size. One of the results of the paper is that $|A/G| < n$ if $G$ is primitive unless $n = 3^{4}$, $5^4$, $3^8$, $5^8$, or $3^{16}$. This bound is sharp when $n$ is prime. In fact, when $G$ is primitive, $|\Out(G)| < n$ unless $G$ is a member of a given infinite sequence of primitive groups and $n$ is different from the previously listed integers. Many other results of this flavor are established not only for permutation groups but also for linear groups and Galois groups.     
\end{abstract}
\maketitle

\section{Introduction}

\subsection{Permutation groups}

Aschbacher and the first author showed \cite{AG2} that if $A$ is a finite permutation group of degree $n$ and $A'$ is its commutator subgroup, then $|A:A'| \leq 3^{n/3}$, furthermore if $A$ is primitive, then $|A:A'| \leq n$. These results were motivated by a problem in Galois theory. For another motivation for the present paper we need a definition. Let $\mathcal{N}$ be a normal series for a finite group $X$ such that every quotient in $\mathcal{N}$ either involves only noncentral chief factors or is an elementary abelian group with at least one central chief factor. Define $\mu(\mathcal{N})$ to be the product of the exponents of the quotients which involve central chief factors. Let $\mu(X)$ be the minimum of the $\mu(\mathcal{N})$ for all possible choices of $\mathcal{N}$. This invariant is an upper bound for the exponent of $X/X'$. In \cite{G} it was shown that if $A$ is a permutation group of degree $n$, then $\mu(A) \leq 3^{n/3}$, furthermore if $A$ is transitive, then $\mu(A) \leq n$, and if $A$ is primitive with $A'' \not= 1$, then the exponent of $A/A'$ is at most $2 \cdot n^{1/2}$. These results were also motivated by Galois theory. In this paper we prove similar statements and obtain corresponding results in Galois theory.    

Let $G$ be a normal subgroup of a permutation group $A$ of finite degree $n$. In this paper the factor group $A/G$ is studied. It is often assumed that $G$ is transitive (this is very natural from the point of view of Galois groups and the results are much weaker without this assumption). Throughout the paper the base of the logarithms is $2$ unless otherwise stated. Our first result is the following.
 
\begin{theorem}
\label{main:0}
Let $G$ and $A$ be permutation groups of finite degree $n$ with $G \vartriangleleft A$. Suppose that $G$ is primitive. Then $|A/G| < n$ unless
$G$ is an affine primitive permutation group and the pair $(n,A/G)$ is $(3^{4},\ort{4}{-}{2}$, $(5^{4},\spl{4}{2})$, $(3^{8},\ort{6}{-}{2})$, $(3^{8},\sort{6}{-}{2})$, $(3^{8},\ort{6}{+}{2})$, $(3^{8},\sort{6}{+}{2})$, $(5^{8},\spl{6}{2})$, $(3^{16},\ort{8}{-}{2})$, $(3^{16},\sort{8}{-}{2})$, $(3^{16},\ort{8}{+}{2})$, or $(3^{16},\sort{8}{+}{2})$. Moreover if $A/G$ is not a section of $\mathrm{\Gamma L}_{1}(q)$ when $n = q$ is a prime power, then $|A/G| < n^{1/2} \log n$ for $n \geq 2^{14000}$.
\end{theorem}

The $n-1$ bound in Theorem \ref{main:0} is sharp when $n$ is prime and $G$ is a cyclic group of order $n$. For more information about the eleven exceptions in Theorem \ref{main:0} and for a few other examples see Section \ref{Section 3}. Note that for every prime $p$ there are infinitely many primes $r$ such that the primitive permutation group $G \leq \mathrm{A \Gamma L}_{1}(q)$ of order $np = qp = r^{p-1}p$ satisfies $|N_{\mathrm{S}_n}(G)/G| = (n-1)(p-1)/p$. It will also be clear from our proofs that the bound $n^{1/2} \log n$ in Theorem \ref{main:0} is asymptotically sharp apart from a constant factor at least $\log_{9}8$ and at most $1$. 

Our second result concerns the size of the outer automorphism group $\Out(G)$ of a primitive subgroup $G$ of the finite symmetric group $\mathrm{S}_n$.   
 
\begin{theorem}
\label{main:1}
Let $G \leq \mathrm{S}_n$ be a primitive permutation group. Then $|\Out (G)| < n$ unless $G$ is an affine primitive permutation group and one of the following holds. 
\begin{enumerate}   
\item $n = 3^{4}$, $G = {(C_{3})}^{4} : (D_{8} \circ Q_{8})$ and $\Out (G) \cong \ort{4}{-}{2}$.

\item $n = 5^{4}$, $G = {(C_{5})}^{4} : (C_{4} \circ D_{8} \circ D_{8})$ and $\Out (G) \cong \spl{4}{2}$.

\item $n = 3^{8}$, $G = {(C_{3})}^{8} : (D_{8} \circ D_{8} \circ Q_{8})$ and $\Out (G) \cong \ort{6}{-}{2}$.

\item $n = 3^{8}$, $G = {(C_{3})}^{8} : (D_{8} \circ D_{8} \circ D_{8})$ and $\Out (G) \cong \ort{6}{+}{2}$.

\item $n = 5^{8}$, $G = {(C_{5})}^{8} : (C_{4} \circ D_{8} \circ D_{8} \circ D_{8})$ and $\Out (G) \cong \spl{6}{2}$.

\item $n = 3^{16}$, $G = {(C_{3})}^{16} : (D_{8} \circ D_{8} \circ D_{8} \circ Q_{8})$ and $\Out (G) \cong \ort{8}{-}{2}$.

\item $n = 3^{16}$, $G = {(C_{3})}^{16} : (D_{8} \circ D_{8} \circ D_{8} \circ D_{8})$ and $\Out (G) \cong \ort{8}{+}{2}$.

\item $n=q^2$ with $q=2^e$, $e > 1$, $G = {(C_{2})}^{2e} : \mathrm{L}_{2}(q)$ and $|\Out (G)| = q(q-1)e$. 

\end{enumerate} 
\end{theorem}

If $G$ is any of the groups in (1)-(7) of Theorem \ref{main:1}, then $\Out (G) \cong N_{\mathrm{S}_{n}}(G)/G$. This indicates why there are only seven exceptional groups in the statement of Theorem \ref{main:1} and not eleven as in the statement of Theorem \ref{main:0}. (For in four cases in Theorem \ref{main:0} the group $A$ has index $2$ in $N_{\mathrm{S}_n}(G)$.)

Next we state an asymptotic version of Theorem \ref{main:1}. For this we need a definition. Let $\mathcal{C}$ be the class of all affine primitive permutation groups $G$ with an almost simple point-stabilizer $H$ with the property that the socle $\mathrm{Soc}(H)$ of $H$ acts irreducibly on the socle of $G$ and $\mathrm{Soc}(H)$ is isomorphic to a finite simple classical group such that its natural module has dimension at most $6$.

\begin{theorem}
\label{main:1.5}
Let $G \leq \mathrm{S}_{n}$ be a primitive permutation group. Suppose that if $n = q$ is a prime power then $G$ is not a subgroup of $\mathrm{A \Gamma L}_{1}(q)$. If $G$ is not a member of the infinite sequence of examples in Theorem \ref{main:1}, then $|\Out(G)| < 2 \cdot n^{3/4}$ for $n \geq 2^{14000}$. Moreover if $G$ is not a member of $\mathcal{C}$, then $|\Out(G)| < n^{1/2} \log n$ for $n \geq 2^{14000}$. 
\end{theorem}

As mentioned earlier, the bound $n^{1/2} \log n$ in Theorem \ref{main:1.5} is asymptotically sharp apart from a constant factor close to $1$.  

The proof of Theorem \ref{main:0} requires a careful analysis of the abelian and the nonabelian composition factors of $A/G$ where $A$ and $G$ are finite groups. For this purpose for a finite group $X$ we denote the product of the orders of the abelian and the nonabelian composition factors of a composition series for $X$ by $a(X)$ and $b(X)$ respectively. Clearly $|X| = a(X) b(X)$. 

The next result deals with $b(A/G)$ in the general case when $G$ is transitive and in the more special situation when $G$ is primitive.  

\begin{theorem}
\label{main:2}
Let $A$ and $G$ be permutation groups with $G \vartriangleleft A \le \mathrm{S}_n$. If $G$ is transitive,
then $b(A/G) \leq n^{\log n}$. If $G$ is primitive, then $b(A/G) \leq {(\log n)}^{2 \log \log n}$.
\end{theorem}   

In order to give a sharp bound for $a(A/G)$ when $G$ is a primitive permutation group, interestingly, it is first necessary to bound $a(A)$ (for $A$ primitive). In 1982 P\'alfy \cite{palfy} and Wolf \cite{wolf} independently showed that $|A| \leq 24^{-1/3} n^{1 + c_{1}}$ for a solvable primitive permutation group $A$ of degree $n$ where $c_{1}$ is the constant $\log_{9}(48 \cdot 24^{1/3})$ which is close to $2.24399$. Equality occurs infinitely often. In fact $a(A) \leq 24^{-1/3} n^{1 + c_{1}}$ holds \cite{pyber} for any primitive permutation group $A$ of degree $n$. Using the classification theorem of finite simple groups we extend these results to the following, where for a finite group $X$ and a prime $p$ we denote the product of the orders of the $p$-solvable composition factors of $X$ by $a_{p}(X)$. 

\begin{theorem}
\label{main:3}
Let $G \leq \mathrm{S}_{n}$ be primitive, let $p$ be a prime divisor of $n$ and let $c_{1}$ be as before. Then $a_{p}(G) |\Out(G)| \leq 24^{-1/3} n^{1 + c_{1}}$.
\end{theorem} 

Wolf \cite{wolf} also showed that if $G$ is a finite nilpotent group acting faithfully and completely reducibly on a finite vector space $V$, then $|G| \leq {|V|}^{c_{2}}/2$ where $c_{2}$ is the constant $\log_{9} 32$ close to $1.57732$. In order to generalize this result we set $c(X)$ to be the product of the orders of the central chief factors in a chief series of a finite group $X$. In particular we have $c(X) = |X|$ for a nilpotent group $X$. The following theorem extends Wolf's result. 

\begin{theorem}
\label{main:3.5}
Let $G \leq \mathrm{S}_{n}$ be a primitive permutation group. Then $c(G) \leq n^{c_{2}}/2$ where $c_{2}$ is as above. 
\end{theorem}    

Some technical, module theoretic results enable us to show that if $G \vartriangleleft A \leq \mathrm{S}_n$ are transitive permutation groups, then $a(A/G) \leq 6^{n/4}$ (see Theorem \ref{abelian-transitive}). In fact, we show that $a(A/G) \leq {4}^{n/\sqrt{\log n}}$ whenever $n \geq 2$ (see Theorem \ref{t1}). This together with Theorem \ref{main:2} give the following.

\begin{theorem}
\label{main:4}
We have $|A:G| \leq {4}^{n/\sqrt{\log n}} \cdot n^{\log n}$ whenever $G$ and $A$ are transitive permutation groups with $G \vartriangleleft A \leq \mathrm{S}_n$ and $n \geq 2$.
\end{theorem}     

For an exponential bound in Theorem \ref{main:4} we can have $168^{(n-1)/7}$ (see Theorem \ref{168}). See \cite[Proposition 4.3]{pyber} for examples of transitive $p$-groups ($p$ a prime) showing that Theorem \ref{main:4} is essentially the best one could hope for apart from the constant $4$. It is also worth mentioning that a $c^{n/\sqrt{\log n}}$ type bound fails in case we relax the condition $G \vartriangleleft A$ to $G \vartriangleleft \vartriangleleft A$. Indeed, if $A$ is a Sylow $2$-subgroup of $\mathrm{S}_n$ for $n$ a power of $2$ and $G$ is a regular elementary abelian subgroup inside $A$, then $|A:G| = 2^{n}/2n$. The next result shows that an exponential bound in $n$ holds in general for the index of a transitive subnormal subgroup of a permutation group of degree $n$.     

\begin{theorem}
\label{main:5}
Let $G\vartriangleleft \vartriangleleft A \leq \mathrm{S}_n$. If $G$ is
transitive, then $|A:G| \leq 5^{n-1}$.
\end{theorem}

The proof of Theorem \ref{main:5} avoids the use of the classification theorem for finite simple groups. Using the classification it is possible to replace the $5^{n-1}$ bound with $3^{n-1}$. It would be interesting to know whether $|A:G| \leq 2^{n}$ holds for transitive permutation groups $G$ and $A$ with $G\vartriangleleft \vartriangleleft A \leq \mathrm{S}_n$. 

We note that the paper contains sharp bounds for $|A:G|$, $b(A/G)$ and $a(A/G)$ in case $A$ is a primitive permutation group of degree $n$ and $G$ is a transitive normal subgroup of $A$. These are $n^{\log n}$ in the first two cases (see the proofs of Theorems \ref{main:8} and \ref{transitive}), and it is $24^{-1/3}n^{c_{1}}$ in the third case (see Corollary \ref{c1}). 

\subsection{Galois groups}
\label{Section 1.2}

As briefly mentioned earlier the above results are also motivated by questions in Galois theory. We need to introduce some definitions. 

Let $k \leq F \leq E$ be three fields such that $k$ is algebraically closed in $E$ and the field extension $E \mid F$ is finite and separable. Let the degree of the extension $E \mid F$ be $n$. Let $L$ be the Galois closure of $E \mid F$. Let $k'$ denote the algebraic closure of $k$ in $L$. Since $L \mid F$ is a Galois extension, so are $L \mid k'F$ and $L \mid E$ (where $k'F$ denotes the compositum over $k$ of the subfields $k'$ and $F$ of $L$). We set $A = \mathrm{Gal}(L \mid F)$, $G = \mathrm{Gal}(L \mid k'F)$, and $H = \mathrm{Gal}(L \mid E)$.

By the definition of $k'$, the Galois group $A$ leaves $k'$ invariant. Since $G$ is the kernel of this action, $G$ is normal in $A$ and $A/G \cong \mathrm{Gal}(k'F \mid F)$. Furthermore $A/G \cong \mathrm{Gal}(k' \mid k)$ (since $A/G$ embeds in $\mathrm{Gal}(k' \mid k)$ which embeds in $\mathrm{Gal}(k'F \mid F)$). 

Since $k'$ is linearly disjoint from both $E$ and $F$, viewed as extensions of $k$, we have that $n = |E : F| = |Ek' : Fk'|$. Since $G \cap H \cong \mathrm{Gal}(L \mid k'E)$, the index of $G \cap H$ in $G$ is $n$. Thus $G$ is transitive on the $A$-set $A/H$.   

We can take $E$ and $F$ to be function fields of $k$-varieties and the extension $E \mid F$ arising from a separable map between varieties. One of the original motivations for studying problems of this sort was the particular case when $E$ and $F$ are function fields of curves over $k$. 

We say that $E \mid F$ is indecomposable if it is a minimal field extension. This is equivalent to saying that $A$ acts primitively on $A/H$. We say that $E \mid F$ is geometrically indecomposable if $Ek' \mid Fk'$ is a minimal field extension, or equivalently $G$ is primitive on $A/H$ (this terminology comes from the case of considering covers of curves). 

We recall a theorem \cite{G} of the first author in this setting. If $E \mid F$ is indecomposable and $k$ is procyclic (its absolute Galois group is procyclic; eg., if $k$ is finite), then $|k':k| < n$. Another theorem of a similar nature had been obtained earlier by Aschbacher and the first author in \cite{AG2}. If $E \mid F$ is indecomposable and $K \mid F$ is an abelian extension (a Galois extension whose Galois group is abelian) with $K \subseteq L$, then $|K : F| \leq n$. 

We now recast two of our results using the notation and assumptions made above. Theorem \ref{main:4} implies the following.

\begin{theorem}
\label{main:6}
$|k':k| \leq {4}^{n/\sqrt{\log n}} \cdot n^{\log n}$ for $n \geq 2$. 
\end{theorem}

Theorem \ref{main:0} can be stated in the following form. 

\begin{theorem}
\label{main:7}
Let $E \mid F$ be geometrically indecomposable. Then $|k':k| < n$ unless the pair $(n,\mathrm{Gal}(k' \mid k))$ is among the eleven exceptions in Theorem \ref{main:0}. Moreover if $\mathrm{Gal}(k' \mid k)$ is not a section of $\mathrm{\Gamma L}_{1}(q)$ when $q = n$ is a prime power, then $|k':k| < n^{1/2} \log n$ for $n \geq 2^{14000}$. 
\end{theorem}

Finally we discuss the situation when the extension $E \mid F$ is indecomposable.

\begin{theorem}
\label{main:8}
If $E \mid F$ is indecomposable, then $|k':k| < n^{\log n}$.  
\end{theorem}

\begin{proof}
If $|A| < n^{1 + \log n}$ or if $A$ is simple, the result follows. Otherwise, by \cite[Theorem 1.1]{maroti}, we have that $A$ is a subgroup of $\mathrm{S}_{m} \wr \mathrm{S}_r$ containing $N = {(\mathrm{A}_{m})}^{r}$, where the action of $\mathrm{S}_m$ is on $k$-element subsets of $\{ 1, \ldots , m \}$ and the wreath product has the product action of degree $n = {\binom{m}{k}}^{r}$. It can be shown that $N$ is the unique minimal normal subgroup of $A$ and so $|A/G| \leq |A/N| < n^{\log n}$. 
\end{proof}

It follows by \cite{GS} that Theorems \ref{main:6}, \ref{main:7}, \ref{main:8} are in fact equivalent to the corresponding group theoretic theorems. 




\subsection{Structure of the paper}   

In this subsection we provide an overview of the present paper. The paper is organized as follows. 

\begin{enumerate}

\item Let $G \vartriangleleft A$ be certain finite groups. In Section \ref{Section 2} a technical result (see Theorem \ref{diag1}) is proved to bound the product $b(A/G)$ of the sizes of the nonabelian composition factors of a composition series of the factor group $A/G$. In specific situations this is an important tool in dealing with nonabelian composition factors. 

\item Section \ref{Section 3} provides some examples of groups $G \vartriangleleft A \leq \mathrm{S}_{n}$ for which $b(A/G)$ is large. These examples demonstrate the sharpness of some of our main results. In particular, Example \ref{example} sheds more light on the eleven exceptions in Theorem \ref{main:0} and also on the seven exceptions in Theorem \ref{main:1}.  
  
\item Let $G \vartriangleleft A \leq \mathrm{GL}(V)$ be finite groups where $V$ is a finite vector space. Suppose that $V$ is an irreducible $G$-module. The purpose of Section \ref{Section 4} is to bound $b(A/G)$ in terms of $n = |V|$. Theorem \ref{linear} states that if $n \geq 3$, then $b(A/G) < (\log n)^{2\log\log n}$. 

\item Section \ref{Section 5} finishes the sole treatment of nonabelian composition factors and this is where the proof of Theorem \ref{main:2} is completed. This is done by proving Theorem \ref{transitive} and Theorem \ref{primitive}. Note that Theorem \ref{transitive} is also used in Section \ref{Section 4}. 

\item For a finite group $G$ and a prime $p$ let $a_{p}(G)$ be the product of the orders of the $p$-solvable composition factors in a composition series for $G$. The key observation (see Theorem \ref{subgroup}) in Section \ref{Section 6} is that for any finite group $G$ and any prime $p$ the invariant $a_{p}(G)$ is bounded from above by the size of a $p$-solvable subgroup of $G$. Section \ref{Section 6} initiates the study of abelian composition factors and establishes a part of Theorem \ref{main:3}. 

\item Section \ref{Section 7} introduces the basic tools for working with abelian composition factors. Lemma \ref{module} collects a few results on the closely related invariant $t_{G}(V)$ which, for a finite group $G$, is defined to be the smallest number $r$ such that every submodule of a $G$-module $V$ can be generated by $r$ elements. This result is used to show Theorem \ref{main:4}. Lemmas \ref{l2} and \ref{l3} are basic tools concerning the solvable group $\Out(S)$ where $S$ is a nonabelian finite simple group. 

\item We consider groups $G$ and $A$ with $G \vartriangleleft \vartriangleleft A \leq \mathrm{S}_{n}$. In Section \ref{Section 8} we prove Theorem \ref{main:5} stating that $|A:G| \leq 5^{n-1}$ provided that $G$ is transitive. However later a stronger bound with a stronger hypothesis is needed (see Theorem \ref{168}). 

\item In Section \ref{Section 9} an important structure theorem is given for primitive linear groups and this is applied for groups acting on a vector space of size larger than $3^{16}$. In particular Theorem \ref{t2} is proved stating that $a(A/G) < n$ where $G \vartriangleleft A$ are primitive permutation groups of degree $n \geq 3^{16}$. Here $a(A/G)$ is the product of the sizes of the abelian composition factors in a composition series for $A/G$.  

\item While Section \ref{Section 9} deals only with abelian composition factors ($a(A/G)$), Section \ref{Section 10} considers both abelian and nonabelian composition factors of finite groups ($a(A/G)$ and $b(A/G)$), still assuming that the number of elements in the vector space (or the permutation domain of an affine permutation group) is larger than $3^{16}$. 

\item In Section \ref{Section 11} we continue the previous investigations. However here the vector space (or the permutation domain) has size at most $3^{16}$. This is the section where the first half of the proof of Theorem \ref{main:0} is completed. Special care is required here because all the eleven exceptions in Theorem \ref{main:0} satisfy $n \leq 3^{16}$.  

\item In Section \ref{Section 12} we bound the size of the outer automorphism group of a primitive permutation group of degree $n$. It turns out that a bound of $n$ is a natural estimate (with few exceptions). Use of deep properties of the first cohomology group are required. In this section Theorem \ref{main:1} is proved which may be considered as an extension of the first half of Theorem \ref{main:0}. 

\item In Section \ref{Section 13} we continue the study of the invariant $a_{p}(G)$ for a finite group $G$ and a prime $p$. We use the methods and results we developed in earlier sections and we complete the proof of Theorem \ref{main:3}. We work with the so-called P\'alfy-Wolf constant. 

\item For a finite group $G$ let $c(G)$ denote the product of the orders of the central chief factors in a chief series for $G$. In Section \ref{Section 13.5} this invariant is considered when $G$ is a primitive permutation group. It again turns out that affine primitive groups $G$ are of central importance. The aim of the section is to prove Theorem \ref{main:3.5}. 

\item Finally, in Section \ref{Section 14} we use slightly different methods to prove asymptotic statements holding for very large values of $n$. This is where the missing parts of the statements of the Introduction are established. These are the second half of Theorem \ref{main:0} and Theorem \ref{main:1.5}. 

\end{enumerate}

\section{Basic results on nonabelian composition factors}
\label{Section 2}

If $G$ is a finite group, define $b(G)$ to be the product of the orders of
all the nonabelian simple composition factors of $G$ in a composition series
for $G$. Two trivial observations that we shall use without comment are:

1. if $G$ is normal in $A$, then $b(A) = b(A/G)b(G)$; and

2. if $A \le B$, then $b(A) \le b(B)$ (choose a normal series for $B$ and
intersect $A$ with this series -- abelian quotients stay abelian and the
nonabelian quotients can only get smaller).

The first lemma of the paper is not used in later parts of the work, nevertheless it is worth mentioning. 

\begin{lemma}
Let $X_{1}$ and $X_{2}$ be two finite groups, $A \leq X_{1} \times
X_{2}$, and $G \lhd A$. For $i = 1, 2$ let $\pi_{i}$ denote the
projection into $X_{i}$. (We consider $\pi_{i}(A)$ and
$\pi_{i}(G)$ as subgroups of $X_{i}$.) Then $$b(A/G) \leq
b(\pi_{1}(A)/\pi_{1}(G)) b(\pi_{2}(A)/\pi_{2}(G)).$$
\end{lemma}

\begin{proof}
Let $K$ denote the kernel of $\pi_{1}$ on $A$. Notice that if $x
\in \pi_{2}(K)$ and $y \in \pi_{2}(G)$ then $[x,y] \in \pi_{2}(G
\cap K)$. Hence $b((\pi_{2}(K) \cap \pi_{2}(G))/\pi_{2}(K \cap G))
= 1$. From this we get
$$b(K/(K \cap G)) = b(\pi_{2}(K)/(\pi_{2}(K \cap G))) = b(\pi_{2}(K)/(\pi_{2}(K) \cap \pi_{2}(G)))
=$$ $$= b(\pi_{2}(K)\pi_{2}(G)/\pi_{2}(G)) \leq
b(\pi_{2}(A)/\pi_{2}(G)).$$ Since $b(A/G) = b(A/GK) b(GK/G) =
b(\pi_{1}(A)/\pi_{1}(G)) b(K/(K \cap G))$, the result follows.
\end{proof}




The next lemma is needed for a technical result (see Theorem \ref{diag1}) for dealing with nonabelian composition factors. 

\begin{lemma}
\label{abelian-diagonal} Let $J \le Y:=X_1\times \cdots \times X_t$ and
assume that $\pi_i(J)=X_i$ for all $i$ (where $\pi_i$ is the projection onto the $i$th
factor). Then $N_Y(J)/J$ is solvable.
\end{lemma}

\begin{proof}
Set $N=N_Y(J)$. Let $M$ be the final term in the derived series of $N$. Let $%
B_i=\ker \pi_i^{\prime}\cap J$ where $\pi_i^{\prime}$ is the projection of $%
Y $ onto the direct product of all but the $i$th term (so $B_i=J \cap X_i$).

Set $R=B_1 \times \cdots \times B_t$. Note that $R
\lhd J$ and that $R$ is also normal in $Y$ since $\pi_{i}(J) =
X_{i}$ for all $i$, whence we may pass to $Y/R$. If we prove the
result in this case, then $MR/R \leq J/R$, and so $M \leq J$.
Hence we may assume from now on that $R=1$. If we prove the
result in this case, then $MR/R \le J/R$, whence $M \le J$.

We induct on $t$. If $t=1$, the result is clear.

Suppose that $t=2$. Since $R=1$, we may identify $J$ as a diagonal subgroup
of $X_{1} \times X_{2}$ and the normalizer $N$ is $J(Z \times Z)$ where $Z=Z(J)$, whence
the result.

So now assume that $t > 2$. By induction, we have $\pi_i^{\prime}(M) \le
\pi_i^{\prime}(J)$, whence $M \le J(N \cap X_i)$. Note that $[N \cap
X_i,X_i]=[N \cap X_i,J] \le X_i \cap J=1$. Thus, $M=[M,M] \le [J(N \cap
X_i),J(N \cap X_i)] \le J$ as claimed.
\end{proof}

Note that the proof shows that the derived length of $N_{Y}(J)/J$ is at
most $t-1$.

We now come to one of our major tools in studying $b(A/G)$.

\begin{theorem}
\label{diag1} Assume that $G \vartriangleleft A \le B = X \wr \mathrm{S}_t = (X_1 \times
\cdots \times X_t).\mathrm{S}_t=Y.\mathrm{S}_t$ and that $G$ acts transitively on $\{X_1, \ldots,
X_t\}$ by conjugation.
Assume that the projection of $N_A(X_i)$ into $X_i$ is $X_i$
(note that $N_A(X_i) \le X_i \times X_i^{\prime}$ for an obvious
choice of $X_i^{\prime}$).
Let $N_i=N_G(X_i)$ and set $M_i$ to be the projection of $N_i$ into $X_i$. Let $K$ be the subgroup of $A$ normalizing each $X_i$. Then
\begin{equation*}
b(A/G) \le b(A/GK)b(X_1/M_1).
\end{equation*}
\end{theorem}

\begin{proof}
We have that $b(A/G)=b(A/GK)b(GK/G)$. So we only need to show that $b(GK/G) \le
b(X_1/M_1)$. Let $M=M_1 \times \cdots \times M_t$. Let $I=J_1 \times \cdots
\times J_t$ where $J_i$ is the projection of $J=K \cap G$ into $X_i$.

Note first that $[K,K \cap M] \le I$ (since $[K,N_i] \le [K,G] \le J$). In
particular, $(K \cap M)/(K \cap I)$ is abelian. By Lemma \ref
{abelian-diagonal}, $(K \cap I)/J$ is solvable. Thus, $(K \cap M)/J$ is
solvable and in particular, $b((K \cap M)/J)=1$. Thus,
\begin{equation*}
b(GK/G)=b(K/J)=b(K/(K\cap M))= b(KM/M).
\end{equation*}

Put $H = \{ y \in Y : [g,y] \in M
\mathrm{\, for \, all \,} g \in G \}$. Notice that $KM \leq H$. From this  
$KM/M \leq H/M \leq Y/M$ follows. Since $G$ permutes the $t$ direct
factors transitively, we see that $H/M$ (and so $KM/M$) is
contained in a full diagonal subgroup of $Y/M = \prod
(X_{i}/M_{i})$. Thus, $b(KM/M) \leq b(X_{1}/M_{1})$ as claimed.
\end{proof}

We will use the following lemma. 

\begin{lemma}
\label{loglog} Suppose that $n=m^t$. Then
\begin{equation*}
t^{\log t} {(\log m)}^{\log\log m} \le (\log n)^{\log\log n}.
\end{equation*}
\end{lemma}

\begin{proof}
It suffices to prove that $(\log t)^2 +(\log\log m)^2 \le (\log\log n)^2$.
The right-hand side is equal to $(\log t + \log\log m)^2$, whence the result.
\end{proof}

\section{Some examples}
\label{Section 3}

In this section, we consider several examples with $G\vartriangleleft A\leq
\mathrm{S}_{n}$ always with $A$ and $G$ transitive. The first example shows that $%
|A/G|=n-1$ may hold even with $G$ primitive and that $b(A/G)$ can be on the
order of $n^{\log n}$ if we only assume that $G$ is transitive (note that
when $n=2^{a}$ for an integer $a$, then it is easy to see that $n^{\frac{1}{2}\log
n}<|\mathrm{L}_{a}(2)|<n^{\log n}$). The third example shows that $b(A/G)$ can be close to
$(\log n)^{2\log \log n}$ even when $G$ is primitive.

\begin{example}
Let $p$ be a prime. Let $V$ be a vector space over ${\F}_p$ of dimension
$a$. Let $A = \mathrm{AGL}_{a}(p)= \mathrm{AGL}(V)$ be the full group of affine transformations
of $V$. Let $G$ be the normal subgroup of translations.

\begin{enumerate}
\item  $G \vartriangleleft A$, $G$ is transitive on $V$ and $A$ is primitive on
$V$;

\item  If $a=1$, then $G$ is primitive on $V$ and $|A/G|=p-1$;

\item  If $a > 1$ and $(p,a) \ne (2,2), (2,3)$, then $b(A/G)=|\mathrm{L}_{a}(p)|$.
\end{enumerate}
\end{example}

\begin{example}
Let $C \vartriangleleft D$ be transitive groups of degree $t$. Let $A=\mathrm{A}_{5}\wr D$
and $G=\mathrm{A}_{5}\wr C$. Then $G \vartriangleleft A$ and they both act primitively on
a set of cardinality $5^{t}$. Then $b(A/G)=b(D/C)$. In particular,
considering the previous example yields examples of primitive $A$ and $G$
with $b(A/G)>t^{\frac{1}{2}\log t}>(\frac{1}{3}\log n)^{(\frac{1}{2}\log
\log n)-2}$.
\end{example}

For the third example and for later use the full symplectic group of dimension $2a$ ($a$ and integer) over the prime field of order $p$ is denoted by $\spl{2a}{p}$. Its order is $p^{a^2} \prod_{i=1}^{a}(p^{2i}-1)$. The group $\spl{2a}{p}$ has a center of size $(2,p-1)$, the greatest common divisor of $2$ and $p-1$, and the corresponding factor group is denoted by $\psp{2a}{p}$. We will also need the orthogonal groups in this paper however only for the field of size $2$ (apart from Lemma \ref{l2}). We denote the full orthogonal groups of dimension $2a$ ($a$ and integer) over the field of order $2$ by $\ort{2a}{\epsilon}{2}$ where $\epsilon = \pm 1$. Their order is $2 \cdot 2^{a(a-1)} (2^{a}-\epsilon) \prod_{i=1}^{a-1}(2^{2i}-1)$. The groups $\ort{2a}{\epsilon}{2}$ have a subgroup of index $2$ which we denote by $\sort{2a}{\epsilon}{2}$.

\begin{example}
\label{example}
Let $p$ be a prime. Let $R$ be a $p$-group of symplectic type -- i.e. $Z(R)$
is cyclic of order $p$ or $4$, $R/Z(R)$ is elementary abelian of order $%
p^{2a}$ for an integer $a$ and $R$ has exponent $p$ for $p$ odd and exponent $4$ for $p=2$. Let
$q$ be a prime power such that $q-1$ is a multiple of $p$ and a multiple of $4$ if $|Z(R)| = 4$. 
Then $R$ embeds in the group $\mathrm{GL}_{p^{a}}(q) = \mathrm{GL}(V)$. Let
$N$ be the normalizer of $R$ in $\mathrm{GL}(V)$ and $Z$ the group of scalars. Then $N/R Z \cong
\spl{2a}{p}$ or $N/R Z \cong \ort{2a}{\epsilon}{2}$ with the latter
possibility occurring when $p=2$ and $Z(R)$ is cyclic of order $2$.
In particular (except for some very small cases),
$b(N/R)=|\psp{2a}{p}|$ or $|\sort{2a}{\epsilon}{2}|$. Let $G=VR$ and $A=VN$. Then
$A$ and $G$ are primitive permutation groups on $V$ and $b(A/G)=|\psp{2a}{p}|$ or
$|\sort{2a}{\epsilon}{2}|$. In particular, when $q=3$, $p=2$ we obtain examples
where $b(A/G)$ has size approximately $(\log n)^{2\log\log n}$.
\end{example}

For the proof of Theorem \ref{main:1} we will need more information about $2$-groups of symplectic type. By \cite[(23.14)]{A} there are, for each positive integer $a$, two extraspecial groups of order $2^{2a+1}$. These are the central product of $a$ copies of $D_{8}$ and the central product of $a-1$ copies of $D_8$ with one copy of $Q_8$. The first can be thought of as an orthogonal space of $+$ type and the other an orthogonal space of $-$ type. The central product of $a$ copies of $D_8$ with one copy of $C_4$ can be thought of as a symplectic space.  

\section{Normalizers of irreducible linear groups -- Nonabelian composition
factors}
\label{Section 4}

In this section, we consider $G \vartriangleleft A \le \mathrm{GL}(V) = \mathrm{GL}_{d}(q)$ where $V$ is a vector space of dimension $d$ over the finite field of order $q$ and want
to bound $b(A/G)$ when $G$ is irreducible on $V$. Of course, this is a
special case of the problem for general pairs of primitive groups -- this is
equivalent to the setup for the case of groups acting primitively on an affine
space.

Recall that a subgroup $A$ of $\mathrm{GL}(V)$ is called primitive if it preserves no
additive decomposition of $V$ (i.e. there is no $A$-invariant collection of
subspaces $V_1,\ldots,V_t$ with $V = \oplus_i V_i$). In particular, this
implies that $A$ is irreducible and every normal subgroup of $A$ acts
homogeneously on $V$ (i.e. any two simple submodules are isomorphic). Recall also the definition of a $p$-group of symplectic type (see Example \ref{example}).


\begin{theorem}
\label{primitive-linear} Let $G \vartriangleleft A\leq \mathrm{GL}(V) = \mathrm{GL}_{d}(q)$. Set $%
n=q^{d} \geq 3$. Assume that $G$ acts irreducibly on $V$ and that $A$ acts
primitively on $V$. Assume that every irreducible $J$-submodule of $V$
is absolutely irreducible for any normal subgroup $J$ of $G$. Then $%
b(A/G) < (\log n)^{2\log \log n}$.
\end{theorem}

\begin{proof}
As we have already noted, every normal subgroup of $A$ acts homogeneously on
$V$. In particular, any abelian normal subgroup acts homogeneously and so is
cyclic by Schur's Lemma. By hypothesis, it must be central. There is no harm
in assuming that $G$ contains all solvable normal subgroups of $A$ (since
that does not affect $b(A/G)$).

We claim that any normal subgroup $R$ of $A$ which is minimal with respect to being non-central is contained in $G$. 

By the first paragraph we may assume that $R$ is not solvable. 

So $Z(R)\leq Z(A)$ consists of scalars and $R/Z(R)$ is characteristically
simple. So either $R$ is a central product of say $t$ quasisimple groups $%
Q_{i}$ (with $Q_{i}/Z(Q_{i})$ all isomorphic) or $R/Z(R)$ is an elementary
abelian $r$-group for some prime $r$. In the second case it follows easily that $R$ is of
symplectic type with $|R/Z(R)|=r^{2a}$ for some $a$, however we may exclude this case in the proof of the claim since $R$ must be non-solvable. 

So $R$ is perfect. Since $G$ acts irreducibly, $C_A(G)=Z(A)$. In particular, $R$
cannot centralize $G$. Suppose that $R$ is not contained in $G$. Then $G
\cap R \le Z(A)$. It follows by the Three Subgroup Lemma \cite[Page 26]{A} that $R=[R,R]$
centralizes $G$, a contradiction.

This proves our claim that every normal subgroup of $A$ which is minimal with respect to being non-central is contained in $G$. 

Let $J_1, \ldots, J_k$ denote the distinct normal subgroups of $A$ that are
minimal with respect to being noncentral in $A$. Let $J=J_1\cdots J_k$ be
the central product of these subgroups. We have shown that $J \le G$. Then $%
C_A(J) = Z(A)$ (for otherwise
the normal subgroup $C_{A}(J)$ of $A$ would contain a normal
subgroup, say $J_{1}$ of $A$ which is minimal with respect to
being non-central, then $J_{1} \leq Z(J)$ which implies that $J_{1}$ is
abelian).

Thus, $A/Z(A)J$ embeds into the direct product of the outer
automorphism groups of the normal subgroups of $A$ which are
minimal subject to being non-central. If $J_{i}$ is such a normal
subgroup and is perfect with $t$ components, then either $t < 5$ and this outer automorphism
group is solvable or $t \ge 5$ and modulo its solvable radical is $\mathrm{S}_t$.

If $J_{i}$ is of symplectic type with $|J_{i}/Z(J_{i})|=r^{2a}$, then this outer
automorphism group has at most one non-solvable composition factor -- $%
\psp{2a}{r}$ or $\sort{2a}{\epsilon}{2}$.

This gives us our upper bound on $b(A/J)$ and so also on $b(A/G)$. Let $W$
be an irreducible constitutent for $J$. Since $A$ is primitive on $V$, it
follows that $J$ acts homogeneously on $V$. It follows by \cite[Lemma 5.5.5, page 205 and Lemma 2.10.1, pages 47-48]{KL} that $W\cong
U_{1}\otimes \cdots \otimes U_{k}$ where $U_{i}$ is an irreducible $J_{i}$%
-module. In particular, if $J_{i}$ is the central product of $t$ copies of a
nonabelian simple group, then $\dim U_{i}\geq 2^{t}$ and if $J_{i}$ is of
symplectic type with $J_{i}/Z(J_{i})$ of order $r^{2a}$, then $\dim U_{i}=r^{a}$.
Moreover, since $U_{i}$ is absolutely irreducible, $r|(q-1)$.

A straightforward computation shows that $\prod_i b(\mathrm{Out}(J_i)) <
(\log n)^{2\log\log n}$ and this finishes the proof. 
\end{proof}

\begin{theorem}
\label{linear} Let $G \vartriangleleft A\leq \mathrm{GL}(V) = \mathrm{GL}_{d}(q)$. Set $n=q^{d} \geq 3$.
Assume that $G$ acts irreducibly on $V$. Then $b(A/G)<
(\log n)^{2\log\log n}$.
\end{theorem}

\begin{proof}
Consider a counterexample with $d$ minimal. We claim that $G$ acts
absolutely irreducibly on $V$. If not, let $E=\mathrm{End}_G(V)$ and let $C$
be the group of units in the field $E$. So $|E|=q^e > q$.

There is no harm in replacing $G$ by $GC$ and $A$ by $AC$ and so
assume $C \le G$.
Let $A_0=C_A(C)$. Then $A/A_0$ is abelian (since it embeds in the
automorphism group of $C$) and so $b(A/G)=b(A_0/G)$. Also, viewing $V$ as a
vector space over $E$, $G$ (and so $A_0$) certainly act irreducibly. Since $%
\dim_E (V) < d$ we obtain a contradiction.

So we assume that $G$ (and so $A$) acts absolutely irreducibly on $V$.

Suppose that $A$ preserves a field extension structure on
$V$ over $\F_{q^e}$ with $e >1$. Let $A_0 = A \cap \mathrm{GL}_{d/e}(q^e)$ and $G_0=G \cap A_0$.
Let $U$ denote $V$ considered as a vector space over $\F_{q^e}$ (and as
an $\F_{q^e}[A_0]$-module). Then $A$ embeds in $\mathrm{GL}_{d/e}(q^e).e$.  Let $W=V \otimes_{\F_q} \F_{q^e}$.
Now $$W \cong \oplus_{\sigma \in {\rm Gal}(\F_{q^e}|\F_q)} U^{\sigma}$$
as an $A_0$-module.  Then $A$ permutes the $U^{\sigma}$.  Moreover,
  $G_0$ acts irreducibly on $U$ (or $G$ acts reducibly on $W$,
  a contradiction to the fact that $V$ is absolutely irreducible
  as a $G$-module).
  Also, $A_0$ acts faithfully on $U$ ($x$ trivial on $U$ implies
  that $x$ is trivial on $U^{\sigma}$ for all $\sigma$, whence
  $x$ is trivial on $W$).
  Then  $b(A/G)=b(A_0/G_0)
  < (\log n)^{2\log\log n}$ contradicts the minimality of $d$ (noting that $n=|U|$).

Suppose that $A$ acts imprimitively on $V$ -- so $V=V_1 \oplus \cdots \oplus
V_t$ with $t > 1$ and $A$ permutes the $V_i$. Note that $G$ must permute the
$V_i$ transitively as well since $G$ is irreducible. Let $K$ be the subgroup
of $A$ fixing each $V_i$. Let $A_i$ be the action of $N_A(V_i)$ on $V_i$ and
define $G_i$ in an analogous way. By Theorem \ref{diag1}, we have $b(A/G) \le b(A/GK)b(A_1/G_1)$. By Theorem \ref{transitive} (see the next section), $b(A/GK) < t^{\log t}$. Now $G_1$ must act
irreducibly on $V_1$ (otherwise $K \cap G$ and so $G$ would act reducibly on $V$) and so by minimality, 
$b(A_1/G_1) < (\log m)^{2\log\log m}$, where $m=|V_1|$. Note that $n=m^t$.
So $$b(A/G) < t^{\log t}(\log m)^{2\log\log m}.$$ The desired conclusion follows from 
Lemma \ref{loglog}.

The remaining case is that $G$ is absolutely irreducibly on $V$, $A$ is
primitive on $V$ and preserves no field extension structure on $V$. Let $J$
be a normal subgroup of $A$. Then $J$ must act homogeneously on $V$ (by the
primitivity hypothesis) and moreover, the irreducible constituents for $J$
must be absolutely irreducible (otherwise the center of $\mathrm{End}_J(V)$
is ${\F}_{q^e}$ for some $e > 1$ and would be normalized by $A$, whence $A$
preserves a field extension structure on $V$). Now the result follows by
Theorem \ref{primitive-linear}.
\end{proof}

\section{Normalizers of transitive and primitive groups -- Nonabelian
composition factors}
\label{Section 5}

We consider the situation $G \vartriangleleft A\leq \mathrm{S}_{n}$. We wish to bound $%
b(A/G)$ when $G$ is transitive and when $G$ is primitive. It is easy to see
that even if one is only interested in the primitive case, one needs an
answer in the transitive case as well.

We first consider the case when $G$ is merely transitive. We have already
used this result in the previous section.

\begin{theorem}
\label{transitive} Let $G$ and $A$ be nontrivial transitive groups with $G \vartriangleleft A \le \mathrm{S}_n$. Then $b(A/G) < n^{\log n}$.
\end{theorem}

\begin{proof}
Suppose the theorem is false and consider a counterexample with $n$ minimal.
First suppose that $A$ is primitive.

Let $E:=F^*(A)$ be the generalized Fitting subgroup of $A$. By the
Aschbacher-O'Nan-Scott theorem, either $E$ is a minimal normal subgroup or $
E = E_1 \times E_2$ with $E_1 \cong E_2$ a direct product of $t$ copies of a
simple nonabelian subgroup $L$ of order $m$.

In the latter case, $n=m^t$, $G$ must contain one of the $E_i$ and by the
structure theorem together with Schreier's conjecture, we see that $b(A/G) \le n(t!)/2 < n^{\log n}$.

In the other cases, $G$ contains $E$ and so we may assume that $G=E$. If $E$
is abelian, then $A/G$ embeds in $\mathrm{GL}_{a}(p)$ with $n=p^a$ for an integer $a$ and so $b(A/G) \le
|\mathrm{L}_{a}(p)| < n^{\log n}$. If $E$ is nonabelian and is the product of $t$
copies of a nonabelian simple group $L$, then either $t \le 4$ and $A/G$ is
solvable or $n \ge 5^t$ and $b(A/G) \le (t!)/2 < n^{\log n}$.

Suppose that $A$ is not primitive. Let $\{ B_{1}, \ldots, B_{t} \}$ be an $A$%
-invariant partition of the underlying set on which $A$ acts. Let $A_i$ denote the action of the stabilizer of $B_i$
in $A$ on $B_i$. Then $A$ embeds in $A_i \wr \mathrm{S}_t$ and $G$ permutes
transitively the subgroups $A_i$. Let $G_i$ denote the action of the
stabilizer of $B_i$ in $G$ on $B_i$. We apply Theorem \ref{diag1} and
induction to conclude that $b(A/G) \le t^{\log t} b(A_1/G_1) \le t^{\log
t}s^{\log s}$ where $n=st$. Thus, the result holds.
\end{proof}

The previous and the next theorem imply Theorem \ref{main:2}.

\begin{theorem}
\label{primitive} Let $G$ and $A$ be primitive groups with $G \vartriangleleft A\leq \mathrm{S}_{n}$ and $n \geq 3$. Then
$b(A/G)<(\log n)^{2\log \log n}$.
\end{theorem}

\begin{proof}
We consider the various cases in the Aschbacher-O'Nan-Scott theorem.

In all cases, $G$ contains $E:=F^*(A)$. The result follows by Theorem \ref
{linear} if $E$ is abelian. So we may assume that $E$ is a direct product of
$t$ copies of a nonabelian simple group $L$ of order $m$. Let $K$ denote the
subgroup of $A$ stabilizing all the components of a minimal normal subgroup of $A$. Clearly $K$ contains $E$. 

Suppose first that $E=E_{1}\times E_{2}$ with $E_{1}\cong E_{2}$ the two
minimal normal subgroups of $A$. In this case $t=2s$ and $n=|E_{1}|=m^{s}$.
Then $GK/K \vartriangleleft A/K$ are transitive subgroups of $\mathrm{S}_{s}$ and so by
Theorem \ref{transitive}, $b(A/GK) \leq s^{\log s}$. On the other hand, $K/E$
is solvable (because it is contained in the direct product of copies of the
outer automorphism group of $L$). Thus, $b(A/G)=b(A/GK) \leq s^{\log s}<(\log
n)^{\log \log n}$.

In the remaining cases, $E$ is the unique minimal normal subgroup of $A$, the groups $
GK/K \vartriangleleft A/K$ are transitive subgroups of $\mathrm{S}_t$ and so, as in the
previous case, we see that $b(A/G) \leq t^{\log t}$. It follows by
the Aschbacher-O'Nan-Scott theorem that $n \ge 5^t$ and so $t^{\log t} < (\log
n)^{\log\log n}$.
\end{proof}

\section{$p$-solvable composition factors of primitive groups}
\label{Section 6}

If $G$ is a finite group, define $a(G)$ to be the product of
the orders of all the abelian (i.e. cyclic) simple composition factors of $G$
in a composition series for $G$. The bounds we will give in this section for $a(G)$ extend naturally to results about $a_{p}(G)$ defined to be the product of the orders of all composition factors of $G$ which are either $p$-groups or $p'$-groups for a given prime $p$. Clearly $a(G) \leq a_{p}(G)$ for any prime $p$ and, by the Odd Order Theorem, $a_{2}(G) = a(G)$.  

It is easy to see that $a_{p}(G)$ (and $a(G)$) is
bounded by the order of some $p$-solvable (solvable) subgroup $S \leq G$ (e.g., this
follows from Theorem \ref{subgroup} below). By a theorem of Dixon \cite{dixon} this
implies (see also Dixon-Mortimer \cite{dixmort}) that $a(G)$ is at most $24^{(n-1)/3}$ for any subgroup $G$ of $\mathrm{S}_{n}$. We state the following more general result.

\begin{proposition}
\label{di} Let $G \leq \mathrm{S}_n$. The product of the orders of all composition factors of $G$ which are not isomorphic to alternating groups of degrees larger than $d \geq 4$ is at most ${d!}^{(n-1)/(d-1)}$. In particular, $a_{2}(G)$ and $a_{3}(G)$ are at most $24^{(n-1)/3}$ and $a_{p}(G) \leq {(p-1)!}^{(n-1)/(p-2)}$ for $p \geq 5$. 
\end{proposition}

\begin{proof}
The first statement follows from \cite[Corollary 1.5]{maroti} by using the argument implicit in the proofs of Theorem \ref{c.permutation} and Lemma \ref{folosleg}.
\end{proof}

Recall that a subgroup $I$ is \emph{intravariant} in a group $G$ if for all
automorphisms $\alpha $ of $G$, the subgroup $I^{\alpha }$ is $G$-conjugate
to $I$.

\begin{lemma}
\label{csu}Let $G$ be a finite group and let $G=G_{0}\vartriangleright
G_{1}\vartriangleright \cdots \vartriangleright G_{r}=1$ be a normal series
with each $G_{i}\vartriangleleft G$. Let $\overline{G_{i}}=G_{i-1}/G_{i}$.
Let $p$ be a fixed prime. Suppose for each $i$ that $I_{i}$ is an
intravariant $p^{\prime }$-subgroup of $\overline{G_{i}}$. Then $G$
has a $p^{\prime }$-subgroup $H$ such that $\left| H\right| \geq
\prod_{i=1}^{r}\left| I_{i}\right|$.
\end{lemma}

\begin{proof}
This is a special case of \cite[Theorem 5.3.17]{suzuki}.
\end{proof}

We also need a consequence \cite[Lemma 2.9]{lipy} of the classification theorem of finite simple
groups.

\begin{lemma}
\label{li}Let $G$ be a nonabelian finite simple group and $p$ a prime. Then $G$ has a solvable
intravariant $p^{\prime }$-subgroup $I$ such that $2|\Out(G)|_{p} \leq |I|$.
\end{lemma}

Now we prove a useful reduction result.

\begin{theorem}
\label{subgroup}
Let $G$ be a finite group and $p$ a prime. Then $a_{p}(G) \leq |S|$ for some $p$-solvable subgroup $S$ of $G$. 
\end{theorem}

\begin{proof}
We may assume that the largest normal $p$-solvable subgroup of $G$ is trivial.  

The socle $L$ of $G$ is a direct product 
$L=L_{1}\times \cdots \times L_{t}$ of nonabelian simple groups $L_{i}$ and $G$ is embedded in $\Aut(L)$.
Denote the kernel of the action of $G$ on the set of subgroups $L_{i}$ by $K$. Then $G/K$
is a permutation group of degree $t$ and $\left| G/K \right|
_{p}\leq 2^{t}$.

Let $I_{1}, \ldots , I_{t}$ be solvable intravariant $p^{\prime }$-subgroups of
maximal orders in the groups $L_{1},\ldots ,L_{t}$. Since $K/L$ is isomorphic to a subgroup of $%
\prod_{i=1}^{t}\Out(L_{i})$, using Lemma \ref{li} we see that
$\left| I_{1}\right| \left| I_{2}\right| \cdots \left| I_{t}\right| \geq
2^{t}\left| K/L\right| _{p}\geq \left| G/L\right| _{p}$.

It is easy to see that the subgroup $I=I_{1}I_{2}\cdots I_{t}$ is
intravariant in $L$.

Consider now a normal series $G =G_{0}\vartriangleright
G_{1}\vartriangleright \cdots \vartriangleright G_{r-1}=L$ such that the groups $G_{i}/L$ form a chief series of $G/L$. Every
$p^{\prime }$-factor $G_{j}/G_{j+1}$ can be considered as an
intravariant $p^{\prime }$-subgroup of itself. Applying Lemma \ref{csu} we
see that $G$ has a $p^{\prime }$-subgroup $S$ whose order is
greater or equal to the product of the orders of these $p^{\prime }$-factors
and $\left| G/L\right| _{p}$. Therefore we have $\left| S \right| \geq
a_{p}(G)$ as required.
\end{proof}








Combining Theorem \ref{subgroup} with well-known results of P\'{a}lfy \cite
{palfy} and Wolf \cite{wolf} one obtains sharp bounds for $a(X)$ for
irreducible linear groups and primitive permutation
groups $X$. Using \cite{HM} we extend these results even further. In the next three results $c_{1} = \log_{9}(48 \cdot 24^{1/3})$ which is close to $2.24399$.

\begin{theorem}
\label{szar}
If $A$ is a finite group acting faithfully and completely reducibly on a finite vector space of size $n$ in characteristic $p$, then 
$a(A) \leq a_{p}(A)\leq 24^{-1/3}{n}^{c_{1}}$.
\end{theorem}

\begin{proof}
By Theorem \ref{subgroup} we know that there is a $p$-solvable subgroup $S$ of $A$ such that $a_{p}(A) \leq |S|$. Moreover, by the construction implicit in the proof of Theorem \ref{subgroup}, we may assume that $O_{p}(S) = 1$ (since $O_{p}(A) = 1$). Thus $S$ can be viewed as a finite group acting faithfully and completely reducibly on a vector space of size $n$. By \cite[Theorem 1.2]{HM} we have $|S| \leq 24^{-1/3}{n}^{c_{1}}$.
\end{proof}

\begin{corollary}
\label{c1}
Let $1\neq G\vartriangleleft A \leq \mathrm{S}_n$ with $A$ primitive. Let $p$ be a prime dividing $n$. Then $
a(A/G) \leq a_{p}(A/G) \leq 24^{-1/3}n^{c_{1}}$.
\end{corollary}

\begin{proof}
We use the Aschbacher-O'Nan-Scott Theorem. The affine case follows from Theorem \ref{szar} by noting that $|G|$ may be taken to be $n$. So assume that $F^{*}(A)$ is nonabelian and it is the direct product of $t$ copies of a nonabelian simple group $L$. By our choice of $p$ the group $L$ (and thus $F^{*}(A)$) is not $p$-solvable. 

If $F^{*}(A)$ is the unique minimal normal subgroup of $A$, then $n \geq {m}^t$ for some divisor $m$ of $|L|$ which is at least the minimal degree of a permutation representation for $L$. In this case $a_{p}(A) \leq {|\Out(L)|}^{t} a_{p}(T)$ where $T$ is a transitive permutation group on $t$ letters. By \cite[Lemma 2.7 (i)]{AG2} we have $|\Out(L)| \leq (2/3) m$, and by Proposition \ref{di} we have $a_{p}(T) \leq n$ (since $m$ can be chosen such that $m \geq p$). These give $a_{p}(A) < (2/3) n^{2}$. This is less than $24^{-1/3}n^{c_{1}}$ unless $n \leq 15$ (and thus $t=1$). If $n \leq 15$, then $a_{p}(A)$ is at most $|\Out(L)| < m = n < 24^{-1/3}n^{c_{1}}$.  

Finally assume that $F^{*}(A)$ is the direct product of two minimal normal subgroups of $A$. In this case $n = l^{t/2}$ where $l = |L| \geq 60$. Again by \cite[Lemma 2.7 (i)]{AG2} and Proposition \ref{di} we find that $a_{p}(A) \leq (2/3) n^{2} < 24^{-1/3}n^{c_{1}}$ for $n \geq 60$ (since $l^{1/2} \geq p$). 
\end{proof}

This immediately implies the following sharp result (which extends the main
result of \cite{palfy} and \cite{wolf}).

\begin{corollary}
\label{cc22}
If $A$ is a primitive permutation group of degree $n$ and $p$ is a prime divisor of $n$, then $a(A) \leq a_{p}(A) \leq 24^{-1/3}n^{1+c_{1}}$.
\end{corollary}

This proves a part of Theorem \ref{main:3}. 

\section{Basic results on abelian composition factors}
\label{Section 7}

Our earlier results on nonabelian composition factors in wreath products do
not help in considering abelian composition factors. We use different
methods for studying abelian composition factors.

The following lemma and its consequences will be crucial in proving Theorem \ref{main:0} on the indices of primitive groups in their normalizers.

If $V$ is a $G$-module over a field, let $t_{G}(V)$ denote the smallest
number $r$ such that every submodule of $V$ can be generated by $r$ elements.

\begin{lemma}
\label{module}Let $H<G$ with $|G:H|=t>1$. Let $W$ be an $H$-module over an
arbitrary field and let $V=W_{H}^{G}$ be the induced module. Then we have the following.
\begin{itemize}

\item[1)] $t_{G}(V)\leq \frac{1}{2}\dim V$. 

\item[2)] If $t\neq 2^{n}$ for any integer $n$, then $t_{G}(V)\leq \frac{1}{3}\dim V$.

\item[3)] If $H \vartriangleleft G$ and $G/H\cong C_{2}^{n}$, then $t_{G}(V)\leq
c_{n}\dim V$ where $c_{n}= \frac{1}{2^{n}} \binom{n}{\left\lfloor n/2%
\right\rfloor }$.

\item[4)] If $t=2^{n}$ for an integer $n$, then $t_{G}(V)\leq \frac{3}{8}\dim V$, unless $H$ is normal
in $G$ and $G/H\cong C_{2}$ or $C_2^2$. Moreover $t_{G}(V)\leq \frac{5}{16}
\dim V$ for $t\geq 32$.
\end{itemize}
\end{lemma}

\begin{proof}
First we prove 1) and 2). By extension of scalars, we may assume that the ground field $k$ is
algebraically closed. Let $p$ be the largest prime dividing $t$ and let $S$
be a Sylow $p$-subgroup of $G$.

Consider the restricted module $V_{S}$. By the Mackey decomposition, this is a direct sum of induced modules of the form ${(W^{g})}_{H^{g} \cap S}^S$ with $g \in G$. Now $p$ divides $t$, so $H^{g} \cap S$ is a proper subgroup of $S$ for all $g \in G$. If we manage to show the proposed bounds for $t_{S}(V_{S})$ then we are finished since $t_{G}(V) \leq t_{S}(V_{S})$. Since $t_{S}$ is subadditive with respective to a direct sum decomposition of $V_{S}$, it is sufficient to bound $t_{S}({(W^{g})}_{H^{g} \cap S}^S)$ for a given $g \in G$. But this means that we may assume that $S=G$ and $H^{g} = H$.

Now let $c\in G$ be an element which does not lie in any conjugate of $H$
and let $C=\left\langle c\right\rangle $. Then, as above, $C \cap H^g$ is
proper in $C$ for all $g\in G$, so by restricting $V$ to $C$ we may
assume that $G$ is a nontrivial cyclic $p$-group.

We can also assume that $W$ is irreducible. Since $H$ is cyclic, $W$ is
$1$-dimensional and the induced module $V$ consists of a single Jordan block,
thus it can be generated by one element. That is, $t_{G}(V) = 1 \leq \frac{1}{p}%
\dim V\leq \frac{1}{2}\dim V$ as required.

Now if $t\neq 2^{n}$, then $p>2$ and 2) follows.

Now we turn to the proof of 3) and 4). If $t=2^{n}$, then let $P$ be the permutation representation of $G$ on the
set of left cosets of $H$ and let $T$ be a Sylow $2$-subgroup of $G$, the
image of $S$ in the permutation representation $P$. Then $T$ is transitive
and so the restricted module $V_{S}$ is simply the induced module $W_{S\cap
H}^{S}$, by the Mackey decomposition.

Instead of $V$ and $W$, we will consider the restricted
modules $V_{S}$ and $W_{S\cap H}$. If $\mathrm{char}(k)\neq 2$, then $V_{S}$ is
semisimple and $t_{G}(V)\leq t_{S}(V)\leq \dim W$ holds. So we can assume
that $\mathrm{char}(k)=2$. Then we can assume that $W_{S\cap H}$ is irreducible, so
the action of $S\cap H$ on $W_{S\cap H}$ is trivial, i.e., $W_{S\cap H}$ is
the trivial $1$-dimensional module. (For this notice that a composition series of $W_{S \cap H}$ corresponds naturally to a series of $\dim W$ submodules of $V$. For any submodule $A$ of $V$ we can view the intersection of $A$ with the members of the previous series. We obtain the claim after summing dimensions corresponding to factor modules of $A$ and by noticing that $c_{n}$ can be viewed as a constant.) Then $V_{S}$ is isomorphic to the
regular representation module of $C_{2}^{n}$. Now using \cite[3.2]{kovnew}
we see that $t_{G}(V)\leq t_{S}(V)\leq c_{n}\dim V$, as required.

In proving 4), we need the following.

\textbf{Claim.} Let $D$ be the permutational wreath product of a regular
elementary abelian $2$-group $R$ and $C_{2}$. If $g$ is an element of order $%
4$ in $D$, then the cycle decomposition of $g$ consists of $4$-cycles.

To see this, write $g$ in the form $g=(a,b)\tau $ where $(a,b)$ is an
element of the base group $R_{1}\times R_{2}$ (here $R_{1}$ and $R_{2}$ are
naturally identified with $R$) and $\tau $ is the involution in the top
group. Then $g^{2}=(a,b)\tau (a,b)\tau =(ab,ba)$. Since $g^{2}\neq 1$, we
see that $ab$ and $ba=(ab)^{-1}$ are both different from the identity, hence
they are fixed point free involutions and so is $g^{2}$ which implies the claim.

Now we will prove 4).

If $T$ itself is not isomorphic to the regular action of $C_{2}^{n}$, then
we prove that $t_{S}(V)\leq \frac{1}{4}\dim V$ from which 4) follows. We
argue by induction. Let $B_{1},B_{2}$ be a $T$-invariant partition. Let $K$
be the stabiliser of the partition. Since $K$ has index $2$ in $T$, $K$ acts
as a transitive group $K_{i}$ on $B_{i}$ so using the inductive hypothesis,
we are done unless $K_{1}$ (or equivalently, $K_{2}$) is isomorphic to the
regular action of $C_{2}^{n-1}$. Then $T$ embeds into the wreath product $%
K_{1}\wr C_{2}$. Now $T$ has an element $g$ of order $4$, otherwise $T$ would
be regular elementary abelian. By our claim, the cycle decomposition of $g$
consists of $4$-cycles. Now using the preimage of $g$ in $G$ we see that $%
t_{G}(V)\leq \frac{1}{4}\dim V$.

If $T$ is isomorphic to $C_{2}^{n}$ then our claim follows from 3).
\end{proof}

Lemma \ref{module} is used in the following result. 

\begin{lemma}
\label{nyuszi}
Let $X_{1}, \ldots, X_{t}$ be finite groups, $X$ their direct product, and let $G$ be an automorphism group of $X$ which permutes the factors transitively. Let $K\leq X$ be a $G$-invariant subgroup, such that for each projection $\pi_{i} $ of $X$ onto $X_{i}$ we have $\pi _{i}(K)\vartriangleleft
\vartriangleleft X_{i}$. Set $J=[G,K]$ (and note that $J$ is normal in $K$
and $G$-invariant). Then $a(K/J)\leq (a(X_{1}))^{t/2}$. If $t\ne 2,4$ then
$a(K/J)\leq (a(X_1))^{3t/8}$ and if $t\geq 17$ then $a(K/J)\leq
(a(X_1))^{t/3}$.
\end{lemma}

\begin{proof}
Let $Y_{1}$ be a minimal characteristic subgroup of $X_{1}$ and let $%
Y=Y_{1}\times \ldots \times Y_{t}$ where $Y_{i}$ are the images of $Y_{1}$
under $G$. Note that
\begin{equation*}
a(K/J)=a(K/J(K\cap Y))a((K\cap Y)J/J)=a(KY/JY)a((K\cap Y)/(J\cap Y))\text{.}
\end{equation*}
So by induction on the length of a characteristic series in $X_1$, we might assume that $X_{i}$ is characteristically simple.

If $X_{i}$ is elementary abelian, then $X$ is an induced module and the
result follows by Lemma \ref{module}. Suppose that $X_{i}$ is a direct
product of isomorphic copies of a nonabelian simple group. Since $\pi
_{i}(K)\vartriangleleft \vartriangleleft X_{i}$, the same is true for each $%
\pi _{i}(K)$, whence $K$ is also a direct product of copies of a nonabelian
simple group. Since $J \vartriangleleft K$, $K/J$ also has the same form, whence
$a(K/J)=1$.

If $t\ne 2,4$ or $t\geq 17$, the same argument applies
(using the stronger conclusions in Lemma \ref{module}).
\end{proof}

We will use Lemma \ref{nyuszi} in the above form, however in one case we will need a refined version. 

\begin{lemma}
\label{imprimitiveexample}
Use the notations and assumptions of Lemma \ref{nyuszi}. Let $t=4$ and for each $i$ with $1 \leq i \leq t$ suppose that $X_{i} = \mathrm{GL}_{2}(3)$. Then $a(K/J) \leq 16^{2} \cdot 3$. 
\end{lemma}

\begin{proof}
In the notation of Lemma \ref{module} we have $t_{G}(V) \leq \frac{1}{2} \dim V$ in general, and $t_{G}(V) \leq \frac{1}{4} \dim V$ for $t=4$ and $\mathrm{char}(k) = 3$. Since $|\mathrm{GL}_{2}(3)| = 16 \cdot 3$, the proof of Lemma \ref{nyuszi} gives $a(K/J) \leq 16^{t/2} \cdot 3^{t/4} = 16^{2} \cdot 3$.
\end{proof}

We will need the following explicit exponential estimate. 

\begin{theorem}
\label{abelian-transitive} Let $G \vartriangleleft A\leq \mathrm{S}_{n}$ with $G$
transitive. Then $a(A/G)\leq 6^{n/4}$.
\end{theorem}

\begin{proof}
If $A$ is primitive, then our statement follows from Corollary \ref{c1} for $n \geq 12$ and from \cite{GAP} for $n \leq 11$.

If $A$ is not primitive, then choose a non-trivial partition $\{ B_{1},\ldots ,B_{t} \}$ that is $A$-invariant with $%
1 < t < n$ maximal. Denote by $A_{1}$ the action of the stabilizer
of $B_{1}$ in $A$ on $B_{1}$ and denote by $K$ the stabilizer of the
partition in $A$. Write $n=st$. Then $$a(A/G)\leq
a(A/KG)a(K/G\cap K)\leq a(A/KG)a(K/\left[ G,K\right] ).$$ 
First suppose that $t$ is different from $2$ and $4$. Then induction and Lemma \ref{nyuszi} yield
$a(A/G) \leq 6^{t/4} \cdot a(A_{1})^{3t/8}$. By Proposition \ref{di}, $a(A_{1}) \leq 24^{(s-1)/3}$ and
so
\begin{equation*}
a(A/G) \leq 6^{t/4} \cdot 24^{(s-1)t/8} < 6^{t/4} \cdot 6^{(s-1)t/4} = 6^{st/4} = 6^{n/4}\text{.}
\end{equation*}
Now let $t = 2$ or $t=4$. Then by \cite[Corollary 1.4]{maroti} we see that $a(A_{1}) \leq 6^{(s-1)/2}$, unless $s = 4$. This and the previous argument using Lemma \ref{nyuszi} give the desired conclusion unless the set of prime divisors of $|A|$ is $\{ 2,3 \}$ and $n = 8$ or $n=16$. But even in this case \cite{GAP} gives the result.  
\end{proof}

An asymptotically better version of Lemma \ref{module} has been obtained by
Lucchini, Menegazzo and Morigi \cite{lucmemo}. The constant in their result has been evaluated by Tracey \cite[Corollary 4.2]{Tracey}.

\begin{lemma}
\label{LMM}
Let $H<G$ with $|G:H|=t>1$. Let $W$ be an $H$-module and let $V=W_{H}^{G}$
be the induced module. Then $t_{G}(V) < 4 \frac{t}{\sqrt{\log t}}\dim W$.
\end{lemma}

Combining this lemma with other ideas above one can easily prove
the following.

\begin{theorem}
\label{t1}
Let $G$ and $A$ be transitive permutation groups of degree $n>1$ with $G \vartriangleleft A$. 
Then $a(A/G)\leq {4}^{n/\sqrt{\log n}}$.
\end{theorem}

\begin{proof}
We use the bound, the notation and the argument of Theorem \ref{abelian-transitive}. By the $6^{n/4}$ bound we see that we may assume that $n > 512$. Also, by Corollary \ref{c1}, it is easy to see that we may assume that $A$ is an imprimitive transitive group. Let $t$ and $s$ be as in the proof of Theorem \ref{abelian-transitive}. By use of Lemma \ref{nyuszi} and Corollary \ref{cc22}, the result follows for $s \geq 32$ as in the proof of Theorem \ref{abelian-transitive}. If $6 \leq s < 32$, then we obtain the result using the fact that $t > 16$. Finally, if $2 \leq s \leq 5$, then $t > 100$ and the bound follows. 
\end{proof}

As pointed out in the Introduction, Theorems \ref{t1} and \ref{transitive} imply Theorem \ref{main:4}.  

We will also use various bounds for the orders of outer automorphism
groups of simple groups.

\begin{lemma} \label{l2}  Let $S$ be a nonabelian finite simple group and suppose that
$S$ has a nontrivial permutation representation of degree
$n$. Then $|\Out (S)| \le 2\log n$ or $S=\mathrm{L}_d(q)$ with 
$d>2$ or $S=\mathrm{P\Omega}_8^+(3^e)$ with $e$ an integer, and $|\Out (S)|\leq 3\log n$. In all cases we have $|\Out (S)| \le 2\sqrt{n}$.
Moreover, $|\Out (S)| \le \sqrt{n}$ unless $S=\mathrm{A}_6$, $\mathrm{L}_2(27)$, $\mathrm{L}_3(4)$
or $\mathrm{L}_3(16)$.
\end{lemma}

\begin{proof} We refer the reader to
Chapter 5 of \cite{KL} for detailed information about
these degrees of permutation representations and
outer automorphism groups.

In most cases we will see in fact that $|\Out (S)|\leq\log n\leq \sqrt{n}$
holds.

If $S$ is sporadic, the Tits group, or an alternating group of degree
other than $6$, then $|\Out(S)| \le 2$ and $n > 4$.
Also, $|\Out(\mathrm{A}_6)|=4\leq 2\log 6$.

So assume that $S$ is a simple group of Lie type in characteristic
$p$ and is not an alternating group. 

Suppose that $S=\mathrm{L}_2(q)$ for some power $q$ of $p$. So $q > 5$ and $q \ne 9$.
The smallest permutation representation is at least $q$
and is generically $q + 1$.  Also, the outer automorphism
group has size $(2,q-1)\log_p q$ which is less than $\log n$
if $p=2$ or $p\geq 5$ and is at most $2\log n$  if $p=3$. If $p=3$
and $q>9$ then $|\Out (S)|=2\log_3q\leq \sqrt{q+1}$  unless
$q=27$.

Next consider the case that $S=\mathrm{L}_d(q), d > 2$.
  We may exclude $\mathrm{L}_3(2) = \mathrm{L}_2(7)$ and  $\mathrm{L}_4(2) = \mathrm{A}_8$.
The order of the outer automorphism group
is $2(d,q-1)\log_p q$ and the smallest permutation
representation is of degree $(q^d-1)/(q - 1)$.

So $\log ((q^d-1)/(q - 1)) > (d-1)\log q \ge
2(d,q-1)\log_p q$ unless one of the following holds.

1. $d|(q-1)$ and $(d-1)/2d \le \log_p q/\log q  = \log_p 2$. In this case
$p \le 7$. If $p=7$, then $d=3$, and if $p=5$, then $d \le 7$.
We have $$|\Out(S)| = \log_p(q^{2d}) < \log (n^{2d/(d-1)}) = 2d \log n/(d-1) \le 3 \log n .$$
Also, note that if $d|(q-1)$, then $\log_p(q^{2d}) < ((q^d-1)/(q-1))^{1/2}$ unless
$d=3$ and $q=4$ or $q=16$.

2. $(d,q-1)=d/2$ and $(d-1)/d \le  \log_p 2$. In this case $p=2$. Also $d/2$ must be odd, whence $d \ge 6$. Thus $|\Out(S)| = \log (q^d) \le (1 + 2/d)\log n
\le (4/3)\log n$. Note also that $(4/3)\log n < \sqrt{n}$ since $d \ge 6$ and $q \ge 4$.

Next consider $U_d(q)$ with $d \ge 3$. We exclude $U_3(2)$ (which is solvable).

Then the outer automorphism group has order $2(d,q+1)\log_p q$.

If $d=3$, then the smallest permutation degree is $q^3+1$
(aside from $q=5$, where it is $50$). If $p > 3$, we see
that $\log n$ dominates $|\Out(S)|$ unless $q=5$.  If $p=3$, then
$|\Out(S)|=2 \log_3 q$ and the result holds.

Finally, if $p=2$ and $q > 2$ is an odd power of $2$, then
$|\Out(S)|=6 \log q$ and $n > q^3$, whence
$|\Out(S)| < 2\log n$.

If $d=4$, then the smallest permutation degree is
$(q+1)(q^3+1)$. We see that $|\Out (S)|=2(4,q+1)\log_p q$
which is less than $\log n$ unless $p=3$ when it is
at most $2\log n$.

If $d\geq 5$, then the smallest permutation representation has
degree roughly $q^{2d-3}$ and we see that $|\Out (S)|\leq\log n$
if $p\geq 5$ and at most $2\log n$ in any case.
In all cases we have $|\Out (S)|\leq \sqrt{n}$.

Suppose that $S=\mathrm{P\Omega}_8^+(q)$, $q$ odd. In this case
$|\Out (S)|=24\log_p q$ and the smallest permutation degree is
$(q^4+1)(q^2+q+1)$. If $p\geq 5$, then we see that $|\Out (S)|\leq 2\log n$.
If $p=3$, then $|\Out (S)|\leq 3\log n$ and in any case
$|\Out (S)|$ is at most $\sqrt{n}$.

In all other cases, $|\Out(S)| \le 8 \log q$
and we see that $|\Out(S)| < \log n$. This completes the proof.
\end{proof}

We remark that the above lemma may be considered as a sharper version of
the observation \cite{AG2} that if $S\ne \mathrm{A}_6$, then $2|\Out (S)|<n$.

A handy consequence of the lemma is that for all nonabelian
finite simple groups $S$ we have $|\Out(S)| \leq  \sqrt[4]{|S|}$ unless
$S=\mathrm{L}_3(4)$. This follows from the known fact that the minimal degree
of a permutation representation of $S$ is less than $\sqrt{|S|}$,
when $|\Out(S)| \leq \sqrt{n}$, and directly in the remaining cases.

We end this section with a result about dimensions versus outer
automorphism groups for simple groups.

\begin{lemma}
\label{l3}
Let $S$ be a nonabelian simple section of $\mathrm{SL}_{n}(p)$ where $p$ is a prime. Then
$|\Out (S)|\leq 4n$.
\end{lemma}

\begin{proof}
For sporadic and alternating groups the result is obvious.

Suppose that $S$ is a group in $\mathrm{Lie}(p')$ over $\mathbb{F}_r$ of
(untwisted) rank $\ell$. By \cite[Lemma 3.1]{lipy} in this case
we have $n\geq\min \{R_p(S),r^{\ell}\}$ where $R_p(S)$ is the
minimal degree of a projective representation of $S$ in characteristic
$p$. Using the lower bounds of Landazuri and Seitz for $R_p(S)$
(slightly corrected in  \cite[Table 5.3A]{KL}) and
\cite[Table 5.1A]{KL}, where values of $|\Out(S)|$ are given, the result
follows by easy inspection.

Suppose now that $S$ is a group in $\mathrm{Lie}(p)$. If the order of $S$ is
divisible by a primitive prime divisor of $p^m-1$ then clearly $n\geq m$
holds. A list of the largest such numbers $m$ is given in
\cite[Table 5.2C]{KL}. Using this we see that in all cases $4m\geq
|\Out(S)|$ holds. This completes the proof.
\end{proof}

\section{Transitive subnormal subgroups}
\label{Section 8}




In this section we prove an amusing variant of Theorem \ref{main:4} for
subnormal transitive subgroups by an elementary argument. 

\begin{theorem1}
\label{subnormal}
Let $G\vartriangleleft \vartriangleleft A \leq \mathrm{S}_n$. If $G$ is
transitive, then $\left| A:G\right| \leq 5^{n-1}$.
\end{theorem1}

\begin{proof}
Consider a counterexample with $n$ minimal.

Let $\{ B_{1},\ldots ,B_{t} \}$ be an $A$-invariant partition, which consists of
blocks of minimal size $k>1$ (if $A$ is primitive, then we take $k=n$).
Denote by $H_{i}$ the action of the stabilizer of $B_{i}$ in $A$ on $B_{i}$.
Denote by $K$ the kernel of the action of $A$ on the set of blocks and by $%
K_{i}$ the action of $K$ on $B_{i}$.

We claim that $\left| K:G\cap K\right| \leq 5^{n-t}$.

By the minimality of $k$ the group $H_{1}$ is primitive.

If $k \leq 7$, then $|H_{1}| \leq \left|\mathrm{S}_k\right| \leq 5^{k-1}$. Furthermore if $k \geq 8$ and $H_{1}$ does not contain
the alternating group $\mathrm{A}_k$, then, by a theorem of Praeger and Saxl \cite{prasax}, we have $\left| H_{1}\right| \leq 4^{k} \leq 5^{k-1}$. This implies that if $k \leq 7$, or if $k \geq 8$ and $H_{1}$ does not contain
the alternating group $\mathrm{A}_k$, then $\left| K:G\cap K\right| \leq \left| K\right| \leq \left| H_{1}\right|
^{t}\leq 5^{n-t}$. Thus we may assume that $k \geq 8$ and that $\mathrm{A}_k \leq H_{1}$.

Since $K_{1} \vartriangleleft H_{1}$, we must have $K_{1}=1,\mathrm{A}_k$ or $\mathrm{S}_k$.

If $K_{1}=1$ then $K=1$.

We may assume that $K$ is a subdirect product of symmetric or alternating groups of
degree $k$. Hence the derived subgroup $K^{\prime }$ is a subdirect product
of alternating groups of degree $k$ and we have $\left| K:K^{\prime }\right|
\leq 2^{t}$. Now $\mathrm{A}_k$ is nonabelian and simple, so the set $\left\{
B_{1},\ldots ,B_{t}\right\} $ has a partition such that $K^{\prime }$ acts
diagonally on the set of blocks which belong to one part of the partition
and $K^{\prime }$ is the direct product of these diagonal subgroups. This 
partition of the set of blocks is clearly $A$-invariant and it follows
that $A$ acts transitively on the set of factors of the direct product, hence $%
K^{\prime }$ is a minimal normal subgroup of $A$.

Let $G=N_{l}\vartriangleleft N_{l-1}\vartriangleleft \cdots \vartriangleleft
N_{1}=A$ be a normal series of length $l$ between $G$ and $A$. Let $j$ be
the largest index with $K^{\prime }\subseteq N_{j}$.

If $j=l$, i.e., $G$ contains $K^{\prime }$, then $\left| K:G\cap K\right|
\leq \left| K:K^{\prime }\right| \leq 2^{t}<5^{n-t}$.

Otherwise $N_{j+1}\cap K^{\prime }$ is normal in $N_{j}$ and it is 
properly contained in $K^{\prime }$. Since $N_{j}$ is transitive, we see
that in fact $K^{\prime }$ is a minimal normal subgroup in $N_{j}$ (and not
just in $A$). Hence $N_{j+1}\cap K^{\prime }=1$, thus $K^{\prime }$
centralizes $N_{j+1}$. Now $N_{j+1}$ is transitive, hence its centralizer is
semiregular. However $K^{\prime }$ acts on $B_{1}$ as $\mathrm{A}_k$ which is
impossible.

This completes the proof of the claim.

Now the permutation group $GK/K$ is transitive on $t$ points and it is
subnormal in $A/K$, so we have $\left| A:GK\right| =\left| A/K:GK/K\right|
\leq 5^{t-1}$. Using the claim we see that
\begin{equation*}
\left| A:G\right| =\left| A:GK\right| \left| GK:G\right| =\left| A:GK\right|
\left| K:G\cap K\right| \leq 5^{n-1}\text{.}
\end{equation*}

Thus the result holds.
\end{proof} 

Using the classification theorem via improved versions \cite{klewha}, \cite{maroti} of the Praeger-Saxl
result one can improve the above bound.
For example the second author \cite{maroti} shows that a primitive group $H$ of
degree $k$ not containing $\mathrm{A}_k$ has order $|H|<c^k$ where $c=2.6$.
This easily implies the bound $|H|\leq 3^{k-1}$. Substituting this bound
in the above proof we obtain that $|A:G|\leq 3^{n-1}$.
It
would be interesting to see whether in fact $\left| A:G\right| \leq 2^{n}$
holds. This would be essentially best possible, since if $n$ is a power of $%
2 $, then the Sylow $2$-subgroup $P$ of $\mathrm{S}_n$ contains a regular
elementary abelian $2$-subgroup of index $2^{n}/2n$ (which is necessarily
subnormal in $P$). Note also that by use of \cite[Corollary 1.4]{maroti} the bound 
$|A:G| \leq 2^{n-1}$ holds in case $A$ is primitive, unless $n=8$, $A = \mathrm{AGL}_{3}(2)$, and $G$ is regular.  

In later sections we will need an improvement (relying on the classification theorem) of the previous bound in a special case.

\begin{theorem}
\label{168}
Let $G\vartriangleleft A \leq \mathrm{S}_n$. If $G$ is
transitive, then $\left| A:G\right| \leq 168^{(n-1)/7}$.
\end{theorem}

\begin{proof}
We may use the notations and the argument of Theorem \ref{main:5}. For this purpose consider a counterexample with $n$ minimal. For convenience set $c = 168^{1/7}$.  

By the remark after the previous theorem, $A$ cannot be a primitive permutation group. Thus $t \geq 2$ and $1 < k < n$.

We may also assume that $|H_{1}| > c^{k-1}$. Furthermore, $H_{1}$ not only does not contain $\mathrm{A}_{k}$ for $k \geq 8$, by the proof of Theorem \ref{main:5}, but it cannot contain $\mathrm{A}_{k}$ even for $k = 5$, $6$, and $7$. (For $k = 6$ we have $|K:K'| \leq 4^{t}$, but this is also sufficient for our purposes.) Taking this a step forward we also see that $H_{1}$ cannot be any of the groups appearing in (ii), (iii) and (iv) of \cite[Corollary 1.4]{maroti}. We conclude that case (i) of \cite[Corollary 1.4]{maroti} holds. Using the bound $|H_{1}| > c^{k-1}$, we may exclude two more groups from the list. Thus $H_{1}$ must be $\mathrm{AGL}_{1}(5)$ for $k = 5$, $\mathrm{AGL}_{3}(2)$ for $k = 8$, $\mathrm{AGL}_{2}(3)$ for $k = 9$, or $\mathrm{AGL}_{4}(2)$ for $k = 16$. 

As in the proof of Theorem \ref{main:5}, we get 
\begin{equation*}
\left| A:G\right| =\left| A:GK\right| \left| GK:G\right| =\left| A:GK\right|
\left| K:G\cap K\right| \leq c^{t-1} \left| K:G\cap K\right| \text{.}
\end{equation*}
But $\left| K:G\cap K\right| \leq b(H_{1}) {(a(H_{1}))}^{t/2} \leq {|H_{1}|}^{t/2} \leq c^{n-t}$, where the first inequality follows from Theorems \ref{diag1} and \ref{nyuszi} and the second inequality from the fact that $t \geq 2$.
\end{proof}

\section{Normalizers of primitive groups -- Abelian composition factors}
\label{Section 9}

We consider the situation $G \vartriangleleft A\leq \mathrm{S}_{n}$, $G$ primitive and
want to bound $a(A/G)$. We first consider the case when the socle of $G$ is
abelian. To deal with this case, we need the following result on primitive
linear groups.

\begin{theorem}
\label{primitive-linear2} 
Let $V$ be a finite vector space of order $n=p^{b}$ defined over a field of prime order $p$. Let $B$ be a subgroup of $\mathrm{GL}(V) = \mathrm{GL}_{b}(p)$ which acts primitively (and irreducibly) on $V$. Let $F$ be a maximal field such that $B$ embeds in $\Gamma L_{F}(V)$%
. Let $|F|=p^{f}$ and let $d=\dim _{F}V$ (so $d=b/f$). Then one of the
following holds.

\begin{enumerate}
\item  $d=1$ and $a(B)\leq (n-1)f\leq (n-1)\log n$; or

\item  $d>1$ and $a(B)< n$ for $n > 3^{16}$.
\end{enumerate}
Furthermore $a(B) < n^{2}/6^{1/2}$ unless $n=9$ and $B = \mathrm{GL}_{2}(3)$. 
\end{theorem}

\begin{proof}
Every normal subgroup of a primitive linear subgroup of $\mathrm{GL}(V)$ acts homogeneously on $V$. In particular, any
abelian subgroup normalized by $B$ acts homogeneously and so is cyclic by
Schur's Lemma. Let $C$ be the subgroup of nonzero elements in $F$ (viewing $%
F $ as a subring of $\mathrm{End}(V)$). Note that $C$ is normalized by $B$
and, for $d=1$, contains the centralizer of $B$. We may replace $B$ by $BC$ and so
assume that $C \le B$.

Let $E=\mathrm{End}_B(V)$ with $q=|E|$. The algebra generated by $C$ is $F$ and $|F|= q^{e}$ for some integer $e$.

Let $B_{0}$ be the centralizer of $C$ in $B$. Note that $C$ is the center of
$B_{0}$. We claim that $B_{0}$ acts irreducibly on $V$ considered as a
vector space over $F$. For let $U$ denote $V$ as a $B_{0}$-module over $F$%
. Then $V^{\prime }:=V\otimes _{E}F\cong \oplus U^{\sigma }$, where the sum
is over the elements of $\mathrm{Gal}(F \mid E)$. Since $B$ acts absolutely irreducibly on $V$
(over $E$), $B$ acts irreducibly on $V^{\prime }$. Note that $B/B_{0}$ acts
regularly on the set $\{ U^{\sigma } \}$ and so $B_{0}$ must act irreducibly on
each $U^{\sigma }$ and so, in particular, on $U$ as claimed.

Note that $a(B)=a(B_0)e$ for $B/B_0$ is cyclic of order $e$.

Let $R$ be a normal subgroup of $B$ contained in $B_{0}$ minimal with
respect to not being contained in $C$. If none exists, then $B_{0}=C$, $d=1$
and $B^{\prime }$ is cyclic and the first conclusion allowed holds. So
assume that this is not the case. Let $W$ be an irreducible $F[R]$-submodule of $V$, which, as an $F[R]$-module, is a direct sum of copies of $W$. Let $F^{\prime }=\mathrm{End}_{R}(W)$.

We claim that $F^{\prime }=F$. The center of the centralizer of $R$ in $%
\mathrm{GL}(V) $ is the group of units of $F^{\prime }$. This is normalized by $B$
and so by the choice of $C$ must just be $C$, whence $F=F^{\prime }$. Let $%
d_{R}$ denote $\dim _{F}W$.

Notice that $R$ cannot be abelian. For if $R$ is abelian, then so is $RC$. But then $RC$ is cyclic by Schur's Lemma and so $RC = C$ by our choice of $C$. This is a contradiction since we chose $R$ not to be contained in $C$. (By this same argument we also see that every characteristic abelian subgroup of $B_{0}$ is central and
contained in $C$.)


So there are two possibilities for $R$.

1. $R$ is of symplectic type with $R/Z(R)$ of order $r^{2a}$ for some prime $%
r$ and integer $a$. Since $Z(R)\leq C$, it follows that $r|q^{e}-1$ and $d_{R}=r^{a}$.
By \cite[Lemma 1.7]{lucmemo2}  in this case $R/Z(R)$ is a completely
reducible ${\F}_r B_{0}$-module under conjugation.

2. $R$ is the central product of $t$ isomorphic quasisimple groups $%
Q_{i},1\leq i\leq t$. Since $R$ acts homogeneously on $V$ and since $%
F^{\prime }=F$, it follows that $W$ is of the form $W_{1}\otimes \cdots \otimes W_{t}$ where $%
W_{i}$ is absolutely irreducible over $F$ (and the tensor product is taken
over $F$). Thus $d_{R}=(\dim _{F}W_{1})^{t}$.

Choose a maximal collection of non-cyclic subgroups described above which pairwise
commute. Denote these by $J_1, \ldots, J_m$. Let $J=J_1 \cdots J_m$ be the
central product of these subgroups.

We next claim that $C_{B_{0}}(J)=C$. Suppose not. By the maximality
condition, any $B$-normal subgroup of $C_{B_{0}}(J)$ minimal with respect to
not being contained in $C$ is one of the $J_{i}$. However, $J_{i}$ is
nonabelian and so is not contained in $C_{B_{0}}(J)$.

In particular $B_{0}/C$ embeds in the direct product of the automorphism
groups of the $J_{i}/Z(J_{i})$. Since $J$ is the central product of the $J_{i}$, $J$
acts homogeneously and $F$ is a splitting field for the irreducible
constituents for each $J_{i}$, it follows that $d=\dim _{F}V\geq \prod d_{i}$
where $d_{i}=d_{J_{i}}$.

Thus, $a(B) \le f(p^f-1) \prod e_i$, where the $e_i$ are defined as follows.

If $J_i$ is of symplectic type with $J_i/Z(J_i)$ of order $r_i^{2\ a_i}$, then
if $B_i$ denotes the (completely reducible) action of $B_{0}$ on
$J_i/Z(J_i)$, we have  $a(B_i)\leq(r_i^{2 a_i})^{2.25}$ by Theorem \ref{szar}.
In this case we set $e_i=r_i^{6.5 a_i}$.

If $J_{i}/Z(J_{i})=L_{1}\times \cdots \times L_{t}\neq 1$ for non-abelian simple groups $L_{i}$, then if
$S_i$ denotes the action of $B_{0}$ permuting the $L_j$, we have
$a(S_i)\leq 24^{(t-1)/3}$ by Proposition \ref{di}. In this case we set
$e_i=|\Out(L_1)|^t 24^{(t-1)/3}$.
Using Lemma \ref{l3} we see that
$$
e_i\leq 4^t\ d_{J_i}\cdot f^t\cdot 24^{(t-1)/3}\leq (d_{J_i})^{4.53}
   f^{[\log d_{J_i}]}.
$$
Altogether we see that $a(B)\leq p^f\cdot f^{1+[\log d]}\cdot d^{6.5}$.
On the other hand $n=p^{fd}$. 

From this, by a tedious calculation, it follows that $a(B) < n$ whenever $n \geq 2^{40}$ (for $d > 1$). With more calculations it is possible to show that $a(B) < n$ whenever $n > 3^{16}$ and $d>1$. 








Finally, consider the last statement of the theorem. By similar calculations as before, it follows that $a(B) < n^{2}/6^{1/2}$ whenever $n \geq 2^{16}$ (even if $d=1$). So assume that $n < 2^{16}$ and also that $d>1$.

If $p^{f} = 2$ then no $J_{i}$ is a group of symplectic type and so a closer look at our previous estimates yields $a(B) < n^{2}/6^{1/2}$ and $a(B) < n$ (if $d > 1$). 

Let $p^{f} = 3$. Then $d \leq 10$, a $J_{i}$ can be a group of symplectic type, but, in this case, we must have $r_{i} = 2$. Using this observation, a simple calculation gives $a(B) < n^{2}/6^{1/2}$ whenever $n > 81$.  

Let $p^{f} = 4$. Then $d \leq 7$, a $J_{i}$ can be a group of symplectic type, but, in this case, we must have $r_{i} = 3$ and $a_{i} = 1$. Using the fact that $|\spl{2}{3}| = 24$, the exponent $6.5$ in the above estimate can be improved in this special case and we get $a(B) < n^{2}/6^{1/2}$ whenever $n > 64$. The same bound holds even in case $d=1$ and $n > 64$.

Let $p^{f} \geq 5$. Here $d \leq 6$ and a very similar argument yields the desired bound.

Thus we only need to check the last statement of the theorem for $n \leq 81$. This was done by GAP \cite{GAP}. 
\end{proof}



\begin{theorem}
\label{linear2} Let $G \vartriangleleft A\leq \mathrm{GL}(V)$ with $|V|=p^{f}=n$. Assume
that $G$ acts irreducibly on $V$. Then either $A$ is metacyclic and $|A/G|<n$
or $a(A/G)<n$ for $n > 3^{16}$.
\end{theorem}

\begin{proof}
Consider a counterexample with $n$ minimal.

If $A$ acts primitively on $V$, then, by Theorem \ref{primitive-linear2}, either $a(A) < n$, or $A^{\prime }$
is cyclic, $A$ embeds in $\mathrm{\Gamma L}_{1}(p^{f})$ and $|A|=a(A)<nf$.

Consider the latter case. Since $G$ acts irreducibly on $V$ over the prime
field, it follows that $|G|\geq f$ (a group of order less than $f$ will not
have an irreducible module of dimension $f$). Thus $|A/G|<n$.

So we may assume that $A$ acts imprimitively.

So $V=V_{1}\oplus \cdots \oplus V_{t}$ with $t>1$ and $A$ permutes the $%
V_{i} $. Note that since $G$ is irreducible, $G$ must permute the $V_{i}$
transitively as well. We may assume that this is done in such a way that $%
V_{1}$ has minimal dimension over $\mathbb{F}_p$. Set $m=|V_{1}|$. Since $%
t>1$ and since $A$ is irreducible on $V$, we have $m>2$.

Let $K$ be the subgroup of $A$ fixing each $V_{i}$. Let $A_{i}$ be the image
of $N_{A}(V_{i})$ acting on $V_{i}$ and define $G_{i}$ similarly. Since $%
V_{1}$ is minimal, it follows that $A_{1}$ acts primitively on $V_{1}$.

Now $a(A/G)\leq a(A/GK)a(K/(G\cap K)$. By Theorem \ref{abelian-transitive}, $%
a(A/GK)\leq 6^{t/4}$. By Lemma \ref{nyuszi} and Theorem \ref{primitive-linear2}%
, $a(K/(G\cap K)\leq a(A_{1})^{t/2} <  (m^{2}/6^{1/2})^{t/2}$ unless $m$ is
$9$ and $A_{1} = \mathrm{GL}_{2}(3)$. Thus, if $m \not= 9$, we have $%
a(A/G) < 6^{t/4}(m^{2}/6^{1/2})^{t/2} = n$.

Assume now that $m = 9$ and $A_{1} = \mathrm{GL}_{2}(3)$. 

By the restriction $n > 3^{16}$, we have $t \not= 2$, $4$. Then Lemma \ref{nyuszi} implies that $a(K/(G\cap K))\leq a(A_1)^{3t/8} = 48^{3t/8}$. Hence $a(A/G)\leq
{\left(6^{1/4} 48^{3/8} \right)}^t < m^t = n$.
\end{proof}

\begin{theorem}
\label{t2}
Let $G$ and $A$ be primitive permutation groups of degree $n$ with $G \vartriangleleft A$. Then
$a(A/G)<n$ for $n > 3^{16}$.
\end{theorem}

\begin{proof}
We consider the various cases in the Aschbacher-O'Nan-Scott Theorem.

In all cases, $G$ contains $E:=F^{\ast }(A)$. The result follows by Theorem
\ref{linear2} if $E$ is abelian. So we may assume that $E$ is a direct
product of $t$ copies of a nonabelian simple group $L$ of order $l$. Let $K$
denote the subgroup of $A$ stabilizing all the components.

Suppose first that $E=E_{1}\times E_{2}$ with $E_{1}\cong E_{2}$ the two
minimal normal subgroups of $A$. In this case $t=2s$ for some integer $s$ and $n=|E_{1}|=l^{s}$.
Then the groups $GK/K \vartriangleleft A/K$ can be considered as transitive subgroups of $\mathrm{S}_{s}$ and so by
Theorem \ref{abelian-transitive}, $a(A/GK)\leq 6^{s/4}$.
By Lemma \ref{nyuszi}, $a(GK/G)=a(K/G\cap K)\leq |\Out(L)|^s$.
Hence $a(A/G)\le n^{1/4}\cdot 6^{s/4}$ unless $L=\mathrm{L}_3(4)$ by a remark after
Lemma \ref{l2}. This is certainly less than $n$ since $l\geq 60$.
The same follows for $L=\mathrm{L}_3(4)$ by direct computation.

In the remaining cases, $E$ is the unique minimal normal subgroup of $A$, the groups $
GK/K \vartriangleleft A/K$ are transitive subgroups of $\mathrm{S}_{t}$ and so as in the
previous case, we see that $a(A/GK)\leq 6^{t/4}$.
Here $n \geq m^t$ where $m$ is at least the minimal degree of a nontrivial permutation
representation of $L$. By Lemmas \ref{nyuszi} and \ref{l2} it follows that $a(GK/G)\leq|\Out (L)|^{t/2}\leq
(2\sqrt{m})^{t/2}$. Hence $a(A/G)\le n^{1/4}\cdot 2^{t/2} \cdot 6^{t/4}$
which is less than $n$ if $m\geq 5$.
\end{proof}

\section{Normalizers of primitive groups -- Sizes}
\label{Section 10}

We continue to consider the situation $G \vartriangleleft A\leq \mathrm{S}_{n}$, $G$ primitive and
want to bound $|A/G|$. We first consider the case when the socle of $G$ is
abelian. To deal with this case, we need the following result on primitive
linear groups.

\begin{theorem}
\label{primitive-linear3} 
Let $V$ be a finite vector space of order $n=p^{b}$ defined over a field of prime order $p$. Let $A$ be a subgroup of $\mathrm{GL}(V) = \mathrm{GL}_{b}(p)$ which acts primitively (and irreducibly) on $V$. Let $F$ be a maximal field such that $A$ embeds in $\mathrm{\Gamma L}_{F}(V)$%
. Let $G$ be a normal subgroup of $A$ which acts irreducibly on $V$. Let $|F|=p^{f}$ and let $d=\dim _{F}V$ (so $d=b/f$). Then $a(A) b(A/G) < f \cdot p^{f} \cdot  d^{2 \log d + 3}$.
\end{theorem}

\begin{proof}
We use the description of the structure of $A$ found in the proof of Theorem \ref{primitive-linear2} (where this group was denoted by $B$). By the fourth paragraph of the proof of Theorem \ref{primitive-linear} we see that $G$ contains every non-solvable normal subgroup of $A$ which is minimal with respect to being non-central. From this the result easily follows.
\end{proof}

We note here that, with little modification and in case $J \not= 1$, the proof of Theorem \ref{primitive-linear} essentially bounds $|A|/(b(G)|J|) = (a(A) b(A/G)) / |J|$, where $J$ is the product of all {\it solvable} normal subgroups of $A$ (satisfying the conditions of Theorem \ref{primitive-linear}) which are minimal with respect to being non-central. We also note that ${(\log n)}^{2 \log \log n}$ is close to $d^{2 \log d}$. However the argument in Theorem \ref{primitive-linear2} is to be used together with Lemma \ref{l3} but excluding Theorem \ref{szar}. 

We continue with a simple lemma.

\begin{lemma}
\label{seged}
Let us use the notations and assumptions of the statement of Theorem \ref{primitive-linear3}. Put $A_{0} = A \cap \mathrm{GL}_{F}(V)$. Suppose that $A$ has a unique normal subgroup $J$ contained in $A_0$ which is minimal subject to being not contained in the multiplicative group $C$ of $F$ viewed as a subset of $\mathrm{End}(V)$. If $|A/G| \geq n$, then $J \leq G$.  
\end{lemma}

\begin{proof}
Let the multiplicative group of the field $K = \mathrm{End}_{G}(V)$ be $L$. 

We may assume that $G$ is not cyclic. Indeed, otherwise $|A| < |L| \cdot |G| < n \cdot |G|$ since $A$ acts on $G$ by conjugation with kernel contained in $L$. 

By the facts that $G$ is not cyclic and a Singer cycle is self centralizing, we must have $df \geq 2$ and $|K| \leq p^{df/r}$ where $r$ is the smallest prime factor of $df$. 

We may also assume that $G$ is metacyclic. Indeed, $G_{0} = G \cap A_0$ is normal in $A$, is contained in $A_0$, thus it may be assumed that $G_{0} \leq C$ is cyclic and thus $G$ is metacyclic. 

By considering the action of $A_0$ on $G$, we see that $|A_{0}| \leq |L| \cdot |G|$ since the kernel of the action is $L \cap A_{0}$, $G_{0} \leq Z(A_{0})$, and $G/G_{0}$ is cyclic. From this we have $|A/G| \leq f \cdot |L| < f \cdot p^{df/r} \leq p^{df} = n$.   
\end{proof}

We next present two useful bounds for $|A/G|$ in terms of $n$. 

\begin{lemma}
\label{segedlemma}
Let $n$, $A$ and $G$ be as in Theorem \ref{primitive-linear3}. Then we have the following. 
\begin{enumerate}
\item $|A/G| < n$ for $n > 3^{16}$; 

\item $a(A) b(A/G) < n^{2}/6^{1/2}$ unless $n=9$ and $A = \mathrm{GL}_{2}(3)$.
\end{enumerate}
\end{lemma}

\begin{proof}
In case $A/G$ is solvable, this follows from Theorems \ref{linear2} and \ref{primitive-linear2}. Thus we may assume that $A/G$ is not solvable. An easy computation using Theorem \ref{primitive-linear3} shows that $|A/G| < n$ for $n \geq 2^{136}$ and $a(A) b(A/G) < n^{2}/6^{1/2}$ for $n \geq 2^{34}$. It is easy to see by the structure of a primitive linear group (see Theorem \ref{primitive-linear2}), that if $p^f = 2$ (where $p$ and $f$ are as in Theorem \ref{primitive-linear3}) and $A/G$ is not solvable, then $n \geq 2^{243}$. Thus we may also assume that $p^{f} \geq 3$ (and $d \geq 4$, where $d$ is as in Theorem \ref{primitive-linear3}). 

Now straightforward calculations using Theorem \ref{primitive-linear3} give $|A/G| < n$ for $n \geq 3^{54}$ and $a(A) b(A/G) < n^{2}/6^{1/2}$ for $n \geq 3^{14}$.

Let us adopt the notations and assumptions of the proof of Theorem \ref{primitive-linear2} (with $B$ replaced by $A$ and $B_0$ replaced by $A_0$). 

Assume that $p^{f} = 3$. Then we may assume that $16 < d \leq 53$ and $d \leq 13$ in the respective cases. A $J_{i}$ can be a group of symplectic type, but, in this case, we must have $r_{i} = 2$. As in Example \ref{example} the normalizer $N_{i}$ in $\mathrm{GL}_{2^{a_{i}}}(3)$ of such a $J_{i}$ satisfies $N_{i}/(J_{i} Z) \cong
\ort{2a_{i}}{\epsilon}{2}$ where $Z$ is the group of scalars. (This is because $|Z(J_{i})|$ must be $2$ since it divides $p^{f}-1$.) A straightforward computation using the structure of $A$ (and $G$) gives the result. 

We only comment on the bound (1) in case $n = 3^{32}$ and when the product $J$ of all normal subgroups of $A$ contained in $A_0$ which are not contained in the multiplicative group, $C$ of $F$ viewed as a subset of $\mathrm{End}(V)$, is solvable. When $J$ is itself a normal subgroup of $A$ contained in $A_0$ which is minimal subject to being not contained in $C$ (a unique such), then $J \leq G$, by Lemma \ref{seged}, and so we find that $|A/G| < n$. Otherwise, if $J$ is a product of more than one $J_{i}$, then $|A| < n$ by the fact that the index of a proper subgroup in $\ort{10}{\epsilon}{2}$, apart from the simple subgroup $\sort{10}{\epsilon}{2}$ whose index is $2$, is at least $495$. 

Assume that $p^{f} = 4$. Then $13 \leq d \leq 42$ and $d \leq 11$ in the respective cases. A $J_{i}$ can be a group of symplectic type, but, in this case, we must have $r_{i} = 3$. We are assuming that $A/G$ is not solvable. As a result, for (2), only the case $d=9$ has to be checked. The bound in (1) is slightly more complicated to establish (but true).

To finish the proof of (2) we may assume that $p^{f} \geq 5$. Then $4 \leq d \leq 9$. Using this information and Theorem \ref{primitive-linear3} we see that $a(A) b(A/G) < f \cdot p^{f} \cdot  d^{2 \log d + 3} < n^{2}/6^{1/2}$. Thus from now on we only consider (1).

Let $p^{f} = 5$. Then we may assume that $11 \leq d \leq 36$. In fact, by use of Theorem \ref{primitive-linear3} we may assume that $d \leq 29$. With a computation similar to the ones above it is possible to deduce (1) in this special case.

We only comment on the case $n = 5^{16}$ and when the product $J$ of all normal subgroups of $A$ contained in $A_0$ which are minimal subject to being not contained in $C$ is solvable. When $J$ is itself a normal subgroup of $A$ contained in $A_0$ which is minimal subject to being not contained in $C$ (a unique such), then $J \leq G$, by Lemma \ref{seged}, and so we find that $|A/G| < n$. Otherwise, if $J$ is a product of more than one $J_{i}$, then $|A| < n$ by the fact that the subgroups $\spl{6}{2} \times \spl{2}{2}$ and $\spl{4}{2} \times \spl{4}{2}$ of $\spl{8}{2}$ are relatively small. 

Let $p^{f} = 7$. We may assume that $10 \leq d \leq 17$ (by use of Theorem \ref{primitive-linear3}). A straightforward computation gives the result. 

Similarly, if $p^{f} = 8$, $9$ or $11$, then we may assume that $d$ satisfies $9 \leq d \leq 16$, $9 \leq d \leq 15$ or $8 \leq d \leq 11$ in the respective cases. Straightforward computations give the result. 

Let $p^{f} = 13$. We may assume that $d = 7$, $8$ or $9$. A straightforward computation gives the result except when $d=8$ and $A$ does not contain a non-solvable normal subgroup which is minimal subject to being not contained in $C$. In this latter case we may proceed as in the case $n = 5^{16}$ described above. 

Let $p^{f} = 16$. We may assume that $d = 7$ or $d=8$. In this case there is nothing to do since we are assuming that $A/G$ is non-solvable.

Let $p^{f} = 17$. We may assume that $d = 7$ and so there is nothing to do. 

Let $p^{f} = 19$. We may assume that $d = 6$. However there is nothing to do since $A/G$ is non-solvable.

By $p^{f} \geq 23$ and Theorem \ref{primitive-linear3}, we have $d = 4$ or $d = 5$. Both these cases can easily be handled using the assumption that $n > 3^{16}$.

This finishes the proof of the lemma.     
\end{proof}

\begin{theorem}
\label{166}
Let $G \vartriangleleft A\leq \mathrm{GL}(V)$ with $|V|=p^{d}=n$. Assume
that $G$ acts irreducibly on $V$. Then $|A/G|<n$ for $n > 3^{16}$.
\end{theorem}

\begin{proof}
By Lemma \ref{segedlemma} we may assume that $A$ acts imprimitively on $V$. By Theorem \ref{linear2} we may also assume that $A/G$ is not solvable.

We may proceed almost as in the relevant paragraph of Theorem \ref{linear}. We may decompose $V$ in the form $V=V_1 \oplus \cdots \oplus
V_t$ with $t > 1$ maximal such that $A$ permutes the $V_i$. Note that $G$ must permute the
$V_i$ transitively as well since $G$ is irreducible. Let $K$ be the subgroup
of $A$ fixing each $V_i$. Let $A_i$ be the action of $N_A(V_i)$ on $V_i$ and
define $G_i$ similarly. 

Now $G_1$ must act irreducibly on $V_1$ and so by Theorem \ref{primitive-linear3} we have the inequality
$a(A_{1}) b(A_1/G_1) < f_{1} \cdot p^{f_{1}} \cdot  d_{1}^{2 \log d_{1} + 3}$, where $m=|V_1| = p^{f_{1}d_{1}}$ for certain integers $f_{1}$ and $d_{1}$. Note that $n=m^t$.

By Theorem \ref{diag1}, we have 
$$|A/G| = a(A/G) b(A/G) \leq a(A/GK)b(A/GK) \cdot a(K/(G \cap K)) b(A_1/G_1).$$

We have $b(A/GK) < t^{\log t}$ by Theorem \ref{transitive}. We also have $a(A/GK) \leq 6^{t/4}$
by Theorem \ref{abelian-transitive}. Thus Lemma \ref{nyuszi}, Theorem \ref{primitive-linear2}, and Theorem \ref{linear} give $$|A/G| \leq 6^{t/4} \cdot t^{\log t} \cdot {a(A_{1})}^{t/2} \cdot b(A_{1}/G_{1}) < 6^{t/4} \cdot t^{\log t} \cdot {(m^{2}/6^{1/2})}^{t/2} \cdot {(\log m)}^{2 \log \log m},$$ provided that $m \not= 9$. However we can improve this bound by use of the inequality 
$${a(A_{1})}^{t/2} b(A_{1}/G_{1}) \leq {f(m)}^{t/2} (a(A_{1})b(A_{1}/G_{1}))/f(m),$$ where $f(m)$ is any upper bound for $a(A_{1})$.
For example if $m \not= 9$ then we get 
$$|A/G| < 6^{t/4} \cdot t^{\log t} \cdot {(m^{2}/6^{1/2})}^{t/2} \cdot ( a(A_{1}) b(A_1/G_1) )/(m^{2}/6^{1/2}).$$ 

First it will be convenient to deal with the case when $t = 2$ or $t = 4$. Then $m \not= 9$. Since $b(A/GK) = 1$, the previous inequality shows that we are done unless $b(A_{1}/G_{1}) \not= 1$. On the other hand, if $b(A_{1}/G_{1}) \not= 1$, then Lemma \ref{segedlemma} gives $$a(A_{1}) b(A_1/G_1) < m^{2}/6^{1/2},$$ provided that $m \not= 9$.

Now let $t$ be different from $2$ and $4$ but at most $16$. 

Assume first that $m \not= 4$. By Lemma \ref{nyuszi} and Theorem \ref{168}, 
$$|A/G| < 168^{(t-1)/7} \cdot {(m^{8/3}/168^{8/21})}^{3t/8} \cdot (a(A_{1}) b(A_1/G_1))/(m^{8/3}/168^{8/21}) <$$ $$< m^{t} \cdot (a(A_{1}) b(A_1/G_1))/(m^{8/3}/168^{8/21}) \cdot 168^{-1/7}$$ (since $m^{8/3}/168^{8/21} > m^{2}/6^{1/2} > a(A_{1})$ for $m \geq 5$ (and $m$ different from $9$), and $m^{8/3}/168^{8/21} >a(A_{1})$ for $m=3$ and $m=9$). Again, we are finished if $b(A_{1}/G_{1}) = 1$. Assume that $b(A_{1}/G_{1}) \not= 1$. If $m = 81$ then we can use GAP \cite{GAP} to arrive to a conclusion. Otherwise it is easy to see that $m \geq 625$. Since $d_{1} \geq 4$, we certainly have
$$a(A_{1}) b(A_1/G_1) < f_{1} \cdot p^{f_{1}} \cdot  d_{1}^{2 \log d_{1} + 3} < m^{8/3}/168^{8/21}$$ for $m \geq 625$, unless possibly if $p^{f} = 2$ or $p^{f} = 3$. If $p^{f} = 2$, then $d_{1} \geq 3^{5}$, by the structure of $A_{1}$, and so the previous inequality holds. We also have the previous inequality in case $p^{f} = 3$ since we may assume by the structure of $A_{1}$ that $d_{1} \geq 8$.

If $m = 4$ then Lemma \ref{nyuszi} and Theorem \ref{168} give us $|A/G| \leq 168^{(t-1)/7} \cdot 6^{3t/8}$. This is not necessarily less than $4^t$, however it is for $t \leq 16$.

Now let $t \geq 17$. 

By Lemma \ref{nyuszi} and Theorem \ref{168}, 
$$|A/G| < 168^{(t-1)/7} \cdot {(m^{3}/168^{3/7})}^{t/3} \cdot (a(A_{1}) b(A_1/G_1))/(m^{3}/168^{3/7}) <$$ $$< m^{t} \cdot (a(A_{1}) b(A_1/G_1))/(m^{3}/168^{3/7}) \cdot 168^{-1/7}$$ (since $m^{3}/168^{3/7} > m^{2}/6^{1/2}$ when $m \geq 4$, and $m^{3}/168^{3/7} > a(A_{1})$ for $m = 3$ and $m=9$). We may assume that $b(A_1/G_1) \not= 1$. Then $m = 81$ or $m \geq 625$ by the structure of $A_{1}$. If $m = 81$ then we have $a(A_{1}) b(A_1/G_1) < m^{3}/168^{3/7}$ by use of GAP \cite{GAP}. If $m \geq 625$ then we arrive to a conclusion by use of three paragraphs up, noting that $m^{8/3}/168^{8/21} < m^{3}/168^{3/7}$ for $m \geq 3$.      
\end{proof}

\begin{theorem}
\label{primitive-size}
Let $G$ and $A$ be permutation groups with $G \vartriangleleft A \leq \mathrm{S}_{n}$. Suppose that $G$ is primitive and 
$|A/G| \geq n$. Then $A$ and $G$ are affine primitive permutation groups and $n \leq 3^{16}$.
\end{theorem}

\begin{proof}
If $A$ is an affine primitive permutation group then the result follows from Theorem \ref{166}. Otherwise we may mimic the proof of Theorem \ref{t2} by noting that we must replace $6^{s/4}$ by $168^{s/7}$ and $6^{t/4}$ by $168^{t/7}$ in the respective cases (due to Schreier's conjecture and Theorem \ref{168}). 
\end{proof}

\section{Small linear groups}
\label{Section 11}

In this section we will finish the proof of the first half of Theorem \ref{main:0}.

Let $G$ and $A$ be permutation groups with $G \vartriangleleft A \leq \mathrm{S}_{n}$. Suppose that $G$ is primitive and 
$|A/G| \geq n$. We must show that the pair $(n,A/G)$ is one of the eleven exceptions in Theorem \ref{main:0}.

By Theorem \ref{primitive-size} it is sufficient to consider affine primitive permutation groups of degrees at most $3^{16}$. 

Let $V$ be a finite vector space of size $n$ with $n \leq 3^{16}$. Opposed to the notation of the statement of the theorem, let $G$ and $A$ be groups such that $G \vartriangleleft A \leq \mathrm{GL}(V)$. Assume that $G$ (and thus $A$) acts irreducibly on $V$. We must classify all possibilities for which $|A/G| \geq n$.   

Let us first assume that $A$ acts primitively on $V$. We use the notations and assumptions of Theorem \ref{primitive-linear2} and its proof (with $B$ replaced by $A$ and $B_0$ replaced by $A_0$). We put $n=p^{b}$ for a prime $p$ and integer $b$ with the property that $A$ is a subgroup of $\mathrm{GL}(V) = \mathrm{GL}_{b}(p)$ acting primitively (and irreducibly) on $V$. Let $F$ be a maximal field such that $A$ embeds in $\Gamma L_{F}(V)$. Let $|F|=p^{f}$ and let $d=\dim _{F}V$ (so $d=b/f$). Let the multiplicative group of $F$, viewed as a subset of $\mathrm{End}(V)$, be denoted by $C$. 

If $d=1$ then Theorem \ref{linear2} gives $|A/G| < n$. Thus assume that $d > 1$.

As in the proof of Theorem \ref{primitive-linear2}, let $J$ be the product of all normal subgroups of $A$ contained in $A_0$ which are minimal subject to not being contained in $C$. 

Assume that $d$ is a prime. Then, by the proof of Theorem \ref{primitive-linear2}, $J$ itself is a normal subgroup of $A$ contained in $A_0$ which is minimal subject to not being contained in $C$. Moreover $J$ is either a quasisimple group or is a group of symplectic type with $|J/Z(J)| = d^2$. In both of these cases we must have $J \leq G$, by Lemma \ref{seged}. 

Assume that $J$ is a quasisimple group. By Lemma \ref{l3} (and the proof of Theorem \ref{primitive-linear2}), we have $|A/G| \leq 4 (p^{f}-1) d f^{2}$. This is less than $p^{df}$ for $d \geq 5$. Assume that $d=2$. If $\mathrm{A}_5$ is a factor group of $J$, then $|A/G| \leq 2 f (p^{f} - 1) < p^{2f}$. Otherwise, by Dickson's theorem on subgroups of $\mathrm{GL}_{2}(p^{f})$, we have $|A/G| \leq 2 f^{2} (p^{f}-1) < p^{2f}$ if $p$ is odd, and $|A/G| \leq f^{2} (p^{f}-1) < p^{2f}$ if $p = 2$. Now assume that $d = 3$. Then, by information from \cite{KL}, we find that $|A/G| \leq 6 f^{2} (p^{f}-1)$. This is smaller than $p^{3f}$ unless $p = f = 2$. If $d = 3$ and $p = f = 2$, then, by \cite{GAP}, we get the desired estimate $|A/G| \leq 4 f (p^{f}-1) = 24 < 64 = n$. 

Let $J$ be a group of symplectic type with $|J/Z(J)| = d^2$ where $d$ is a prime. Then $|A/G| \leq |\spl{2}{d}| f (p^{f}-1) < d^{3} f (p^{f}-1) < n$ for $d \geq 5$, and also for $d = 3$ and $p^{f} > 4$. If $d = 3$ and $p^{f} = 4$, then $|A/G| \leq d^{3} f = 54 < 64$. Let $d = 2$. It is then easy to see that $|A/G| \leq 6 f ((p^{f}-1)/2) < p^{2f}$ since $p > 2$.          

From now on we assume that $d$ is not a prime and larger than $1$. 

In this paragraph let $p^{f} = 2$. By the structure of $A$ described in the proof of Theorem \ref{primitive-linear2} we know that all normal subgroups of $A$ contained in $A_0$ and minimal with respect to being not contained in $C$ are non-solvable. Moreover $J$ has at most two non-abelian simple composition factors, since $d \leq 25$. By this, we immediately see, as in the proof of Theorem \ref{primitive-linear2}, that $|A/G| \leq 32 d$. This is less than $2^d$ unless $d \leq 8$. If $d = 6$ or $8$, then $|A/G| \leq 4d < 2^d$. For $d = 4$ the result follows by \cite{GAP}.

From now on we assume that $p^{f} > 2$. 

In this paragraph we deal with the cases when $d = 6$, $10$, $14$, or $15$. In these cases $d$ is a product of two primes $r_{1}$ and $r_{2}$. First suppose that $J$ is not solvable. If $A$ has no solvable normal subgroup contained in $A_{0}$ which is minimal with respect to being noncentral, then it is easy to see that $|A/G| \leq f (p^{f}-1) \cdot 4^{2} f^{2} d < p^{fd}$ since $p^{f} > 2$. Otherwise we get $|A/G| \leq f^{2} (p^{f}-1) \cdot 4 r_{1} \cdot {r_{2}}^{5}$ (for a certain choice of $r_1$ and $r_2$). This is always less than $p^{df}$ unless $d=6$ and $p^{f} = 3$ or $4$. If $d = 6$ and $p^{f} = 3$, then in the previous bound we must have $r_{2} = 2$ and thus $|A/G| < n$. If $d = 6$ and $p^{f} = 4$, then we must have $r_{1} = 2$ and $r_{2} = 3$. In this special case we can modify our bound to $|A/G| \leq f (p^{f}-1) \cdot 2 \cdot 3^{5} = 12 \cdot 3^{5} < 4^{6} = n$. Thus we may assume that $J$ is solvable. In this case $d$ divides $p^{f}-1$, and since $n \leq 3^{16}$, we are left to consider only the case $d = 6$ and $p^{f} = 7$ or $13$ when $|A| < n$.      

We are left to consider the cases when $d = 4$, $8$, $9$, $12$, or $16$. 

Let $d=4$. 

First assume that $J$ is solvable and it is the unique normal subgroup of $A$ contained in $A_0$ which is minimal subject to being not contained in $C$. By Lemma \ref{seged} we may assume that $J \leq G$. For $p^{f} \geq 7$ we can bound $|A/G|$ by $f ((q^{f}-1)/2) |\spl{4}{2}| = 360 f (q^{f}-1) < p^{4f}$. We are left to consider the cases when $p^{f} = 3$ and $p^{f} = 5$. If $p^{f} = 3$, $d = 4$ and $|A/G| \geq 81$, then $(n,A/G) = (3^{4},\ort{4}{-}{2})$, while if $p^{f} = 5$, $d = 4$ and $|A/G| \geq 625$, then $(n,A/G) = (5^{4}, \spl{4}{2})$. Now assume that $J$ is solvable and it is the product of two normal subgroups, say $J_{1}$ and $J_{2}$ of $A$ contained in $A_{0}$ which are minimal subject to being not contained in $C$. If $f = 1$ then $G$ contains one (if not both) of these normal subgroups, say $J_{1}$. Furthermore, since $J_{1}$ is not irreducible on $V$, the irreducible group $G$ properly contains $J_{1}$. Thus $|A/G| \leq 4 \cdot 36 \cdot ((p^{f}-1)/2) \cdot (1/2) = 36 (p-1) < p^4$. We may now assume that $f \geq 2$ (and also that $p$ is odd). In this case we only use the fact that $|G| \geq 4$ to conclude that $|A/G| \leq f (p^{f}-1) 16 \cdot 36 \cdot (1/4) = 144 f (p^{f}-1)$. We already know from the same paragraph that this is less than $p^{4f}$ for $p^{f} \geq 9$.    

Secondly assume that $A$ has no solvable normal subgroup contained in $A_0$ which is minimal subject to not being contained in $C$. In this case $J$ has at most two non-abelian composition factors and so $|A/G| \leq f^{3} (p^{f}-1) \cdot 4^{3} \cdot 2 = 128 f^{3} (p^{f}-1)$, by the second half of the proof of Theorem \ref{primitive-linear2}. From this we get $|A/G| < p^{4f}$ unless possibly if $p^{f} = 3$, $4$, $8$, $9$ or $16$. When $J$ has a unique non-abelian composition factor, then we may sharpen our bound to $|A/G| \leq 16 f^{2} (p^{f}-1)$, and this is smaller than $p^{4f}$ for the remaining five values of $p^f$. Thus $J$ has exactly two non-abelian composition factors. In this case we can apply Dickson's theorem on subgroups of $\mathrm{GL}_{2}(p^{f})$ to refine our bound on $|A/G|$ even further. This is $8 f^{3} (p^{f}-1)$ which is smaller than $p^{4f}$ for the remaining five values of $p^f$.     

Thirdly there are two normal subgroups of $A$ contained in $A_0$ which are minimal subject to not being contained in $C$. One is $J_{1}$, a symplectic $2$-group, and one is $J_{2}$, a quasisimple group. In this case we have $|A/G| \leq f (p^{f}-1) \cdot 2f \cdot 24 = 48 f^{2} (p^{f}-1)$. This is less than $p^{4f}$ where $p > 2$, unless $p^{f} = 3$. But $p^{f} = 3$ cannot occur in this case since $J_{2} \leq \mathrm{GL}_{2}(3)$ is solvable.

From now on let $d$ be $8$, $9$, $12$ or $16$. 

In this paragraph suppose that $A$ has no solvable normal subgroup contained in $A_0$ which is minimal subject to not being contained in $C$. In this case the number, say $r$ of non-abelian composition factors of $J$ is at most $4$. If $r = 4$ then $d=16$ and so $p^{f} = 3$. In this case it is easy to see that $|A/G| \leq 98304 f^{5} (p^{f}-1) < 3^{16}$. Let $r = 3$. Then $d = 8$ or $d \geq 12$. In the first case we can use Dickson's theorem to conclude that a quasisimple subgroup $Q$ of $\mathrm{GL}_{2}(p^{f})$ satisfies $|\Out(S/Z(S))| \leq 2f$. This implies that $|A/G| \leq 48 f^{4} (p^{f}-1) < p^{8f}$. In case $d \geq 12$ we can use our usual bound $|A/G| \leq f^{4} (p^{f}-1) \cdot 4^{3} \cdot 16 \cdot 6 = 6144 f^{4} (p^{f}-1) < p^{12f}$. Finally let $r \leq 2$. Then $|A/G| \leq 512 f^{3} (p^{f}-1)$. This is less than $p^{fd}$ for $d \geq 8$ (and $p^{f} > 2$). 

In the remaining cases $A$ has a solvable normal subgroup contained in $A_0$ which is minimal subject to not being contained in $C$. This implies that the greatest common divisor of $d$ and $p^{f} - 1$ is larger than $1$. This, the above, and the fact that $n \leq 3^{16}$ imply that the only cases to deal with are the following: $d = 8$ and $p^{f} = 3$, $5$, $7$, $9$; $d = 9$ and $p^{f} = 4$, $7$; $d = 12$ and $p^{f} = 3$, $4$; and $d = 16$ and $p^{f} = 3$.      

Let $d=8$. We may assume that $A$ has a solvable normal subgroup contained in $A_0$ which is minimal subject to not being contained in $C$. First suppose that $J$ is not solvable. Then $J$ has one or two non-abelian composition factors. Such a composition factor can be considered as a subgroup of $\mathrm{L}_{2}(p^{f})$ or of $\mathrm{L}_{4}(p^{f})$. In the first case we must have $p^{f} \geq 5$. Suppose $J$ has exactly one non-abelian composition factor. If this is a subgroup of $\mathrm{L}_{2}(p^{f})$, then, by Dickson's theorem, we have the estimate $|A/G| \leq f(p^{f}-1) \cdot 2f \cdot |\spl{4}{2}| \cdot 2^{4} < p^{8f}$ for $p^{f} \geq 5$. If this is considered as a subgroups of $\mathrm{L}_{4}(p^{f})$, then $|A/G| \leq f(p^{f}-1) \cdot 16f \cdot |\spl{2}{2}| \cdot 4 < p^{8f}$. Finally, if $J$ has exactly two non-abelian composition factors, then these must be subgroups of $\mathrm{L}_{2}(p^{f})$, and we have $|A/G| \leq f (p^{f}-1) \cdot {(2f)}^{2} \cdot 2 \cdot |\spl{2}{2}| \cdot 4 < p^{8f}$ for $p^{f} \geq 5$. Thus we may assume that $J$ is solvable. First assume that $p^{f}=9$. If $A$ has more than one normal subgroup contained in $A_0$ which is minimal subject to not being contained in $C$, then $|A| \leq 2 \cdot 8 \cdot |\spl{4}{2}| |\spl{2}{2}| \cdot 2^6 < 9^8$. Otherwise we may assume that $J \leq G$, by Lemma \ref{seged}, and so $|A/G| \leq 2 \cdot 8 \cdot |\spl{6}{2}| < 9^8$. We may now assume that $p^{f} = 3$, $5$, or $7$. In all of these cases $A_{0} = A$. First suppose that $A$ has more than one normal subgroup which is minimal subject to not being contained in $C$. If $p^{f} \not= 3$, then $|A/G| \leq 16 \cdot 3 \cdot |\spl{4}{2}| |\spl{2}{2}| < p^{8f}$. Let $p^{f} = 3$. If $|G| \geq 16$, then $|A/G| \leq 8 \cdot |\ort{4}{-}{2}| |\ort{2}{-}{2}| < 3^{8}$. Otherwise $|G \cap J| = 8$ and in fact $|G| = 8$. But such a group $G$ cannot act irreducibly on $V$. We conclude that $J$ is the unique normal subgroup of $A$ which is minimal subject to not being contained in $C$. Thus $J \leq G$. If $p^{f} = 7$, then $|A/G| \leq |\ort{6}{-}{2}| < 7^8$. Let $p^{f} = 5$. Assume that $A$ is the full normalizer of $J$ in $\mathrm{GL}(V)$. If $G = J$, then $(n,A/G) = (5^{8},\spl{6}{2})$. Otherwise, since $A/J \cong \spl{6}{2}$ is simple, $G = A$. Thus we may assume that $A/J$ is a proper subgroup of $\spl{6}{2}$. By \cite[pp. 319]{dixmort}, we have $|A/G| \leq |A/J| \leq |\spl{6}{2}|/28 < 5^8$. We remain with the case $p^{f} = 3$. If $G = J$ and $J \leq A$ has index at most $2$ in the full normalizer of $J$ in $\mathrm{GL}(V)$, then $|A/G| > n$ and $(n,A/G) = (3^{8},\ort{6}{-}{2})$, $(3^{8},\sort{6}{-}{2})$, $(3^{8},\ort{6}{+}{2})$ or $(3^{8},\sort{6}{+}{2})$. Suppose now that $J \leq A$ has index larger than $2$ in the full normalizer of $J$ in $\mathrm{GL}(V)$. Since $\sort{6}{+}{2} \cong \mathrm{A}_{8}$ and $\sort{6}{-}{2} \cong \mathrm{U}_{4}(2)$ are simple groups with minimal index of a proper subgroup $8$ and $27$ respectively (for the latter see \cite[pp. 317]{dixmort}), we immediately get $|A/G| \leq |A/J| < 3^8$ in the remaining cases. 
     
Let $d=9$. We may assume that $A$ has a solvable normal subgroup contained in $A_0$ which is minimal subject to not being contained in $C$. If $J$ is non-solvable, then, by the structure of $A$ (and $G$), $|A/G| \leq f (p^{f}-1) \cdot 4 \cdot 3f \cdot 3^{2} \cdot |\spl{2}{3}| < p^{9f}$. Thus $J$ is solvable. If $p^{f} = 7$ then an easy computation yields $|A| \leq 6 \cdot 81 \cdot |\spl{4}{3}| < 7^9$. We assume that $p^{f} = 4$. Now
$J$ is the product of one or two normal subgroups of $A$ not contained in $C$. If one, then we may assume by Lemma \ref{seged} that $J \leq G$. In this case we get $|A/G| \leq 2 \cdot |\spl{4}{3}| < 4^9$. In the other case we get $|A| \leq 6 \cdot 81 \cdot {|\spl{2}{3}|}^{2}$. Since $|G| > 2$, we see that $|A/G| < 4^9$.       

Let $d=12$. We may again assume that $A$ has a solvable normal subgroup contained in $A_0$ which is minimal subject to not being contained in $C$. We may also assume that $J$ is not solvable since $p^f$ is $3$ or $4$. If $J$ has one non-abelian composition factor, then $|A/G| \leq f (p^{f}-1) \cdot 4 \cdot 6 f \cdot 16 \cdot |\ort{4}{-}{2}| = f^{2} (p^{f}-1) \cdot 46080 < p^{12f}$ for both $p^{f} = 3$ and $p^{f} = 4$. Finally, if $J$ has two non-abelian composition factors, then $|A/G| \leq f (p^{f}-1) \cdot 4^{2} \cdot 4 \cdot f^{2} \cdot 2 \cdot 9 \cdot |\spl{2}{3}| = f^{3} (p^{f}-1) \cdot 27648 < p^{12f}$ for both $p^{f} = 3$ and $p^{f} = 4$. 

Let $d=16$. Then $p^{f} = 3$ and $A_{0} = A$. From the above we may assume that $A$ has a solvable normal subgroup which is minimal subject to not being contained in $C$. Assume first that $J$ is not solvable. Since $\mathrm{GL}_{2}(3)$ is solvable and $3$ does not divide $16$, we know that $J$ has a unique non-abelian composition factor and this can be considered as a subgroup of $\mathrm{L}_{4}(3)$. From this we arrive to a conclusion by $|A/G| \leq f (p^{f}-1) \cdot 4 \cdot (4f) \cdot |\ort{4}{-}{2}| \cdot 2^{4} < 3^{16}$. Thus we may assume that $J$ is solvable. First assume that $A$ contains more than one normal subgroup which is minimal subject to not being contained in $C$. In this case $|A/G|$ is at most $2^{6} \cdot |\ort{6}{-}{2}||\ort{2}{-}{2}| < 3^{16}$. Thus we may assume that $J$ is the unique normal subgroup of $A$ which is minimal subject to not being contained in $C$. This implies that $J \leq G$. If $G = J$ and $J \leq A$ has index at most $2$ in the full normalizer of $J$ in $\mathrm{GL}(V)$, then $|A/G| > n$ and $(n,A/G)$ is $(3^{16},\ort{8}{-}{2})$, $(3^{16},\sort{8}{-}{2})$, $(3^{16},\ort{8}{+}{2})$ or $(3^{16},\sort{8}{+}{2})$. Now $\sort{8}{\epsilon}{2}$ are simple groups with the property that every proper subgroup has index at least $119$ (see \cite[pp. 319--320]{dixmort}). This implies that if $J \leq A$ has index larger than $2$ in the full normalizer of $J$ in $\mathrm{GL}(V)$, then $|A/G| \leq |A/J| < 3^{16}$.

We may now assume that $A$ acts imprimitively on $V$. By the proofs of Theorems \ref{linear2} and \ref{166} we see that we may assume, in the notations of these proofs, that $m = 9$ and $t = 2$ or $t=4$. In these cases $n = 3^4$ or $n = 3^8$. If $n = 3^4$, then \cite{GAP} gives $|A/G|< n$. So assume that the second case holds. The group $A$ is clearly solvable and so by Lemmas \ref{imprimitiveexample} and \ref{abelian-transitive} we have $|A/G| \leq 6 \cdot 16^{2} \cdot 3$. This is less than $3^{8}$. 

This completes the proof of the first half of Theorem \ref{main:0}.   

\section{Normalizers and outer automorphism groups of primitive groups}
\label{Section 12}

In this section we prove Theorem \ref{main:1}.

Let $G$ be a primitive permutation group of degree $n$. Assume first that the generalized Fitting subgroup $E=F^*(G)$ of $G$ is
nonabelian. By \cite[Lemma 1.1]{rose} $\Aut(G)$ has a natural embedding into
$\Aut(E)$ and it acts transitively on the components of $E$
(this is also true if $G$ has two minimal normal subgroups). The bound follows as in the proof of the bound for $|A/G|$.

So suppose that $F^{\ast }(G)=V$ is abelian. If $V$ is central, then $n$ is a prime, $%
G$ is cyclic and ${\mathrm{Out}}(G)$ is cyclic of order $n-1$. So assume
that $Z(G)=1$. In this case the centralizer of $G$ in $\mathrm{S}_{n}$ is $1$.

\begin{lemma} 
\label{cohomology}
Let $G$ be a primitive affine permutation group
of degree $n$ and $H$ a point stabilizer. Let $V=F(G)=F^*(G)$.
Then $|\mathrm{Aut}(G):N_{\mathrm{S}_n}(G)|=|H^1(H,V)|$.
\end{lemma}

\begin{proof}
Let $N$
be the normalizer of $G$ in $\mathrm{S}_{n}$, then $N$ embeds into $A=\Aut(G)$.
Moreover an element $\varphi \in A\;$is in $N$ exactly when the image of a
point stabilizer $H$ of $G$ under $\varphi $ is also a point stabilizer
\cite[4.2B]{dixmort}.

It is easy to see that $A$ acts transitively on the complements of $V$ hence
it acts transitively on the $G$-conjugacy classes of such complements. It
follows that $\left| A:N\right| $ equals the number of $G$-conjugacy classes
of complements, that is, $\left| H^{1}(H,V)\right| $.
\end{proof}

Assume first that $H^1(H,V)= \{ 0 \}$. Then $\Out (G) \cong N_{\mathrm{S}_n}(G)/G$. In this case we may apply Theorem \ref{main:0}. We see that if $A$ is a permutation group of degree $n$ and $G$ is a primitive normal subgroup in $A$, then $|A/G| < n$ unless one of eleven cases holds. From these eleven exceptional pairs, $(A,G)$, the group $A$ is the full normalizer of $G$ in $\mathrm{S}_n$ in exactly seven cases. These pairs give rise to the seven exceptions in the statement of Theorem \ref{main:1}. In order to characterize the seven groups $G$, the reader is referred to Example \ref{example} and the paragraph that follows it.  

From now on we assume that $H^1(H,V) \ne \{ 0 \}$.
So $G=VH$ with $H$ acting faithfully and irreducibly on
$V$. We must show that if $|\Out (G)| \geq n$, then (8) holds (in the statement of Theorem \ref{main:1}). 

We first point out the following result in \cite[2.7 (c)]{AG1}.
See also \cite{Gu}.

\begin{lemma} 
\label{AGlemma}
Suppose that $H$ is a finite group acting
irreducibly and faithfully on $V$ with $H^1(H,V)\ne 0$.
Then $H$ has a unique minimal normal subgroup $N$ of the form $L_1 \times \cdots
\times L_t$ with $L_i \cong L$ a nonabelian simple group. Set
$L_{i'} = C_N(L_i)$ (the product of all the other $L_j$). Moreover,
$V = \oplus_i V_i$ where $V_i=[L_i,V] = C_V(L_{i'})$
and $|H^1(H,V)| \le |H^1(L_1,W)|$, where $W$ is any nontrivial
irreducible $L_i$-submodule of $V_1$.
\end{lemma}

Now we obtain bounds on the size of outer automorphism groups
for such groups.

First an example. Let $q=2^e, e > 1$ and  $H= \mathrm{L}_{2}(q)$.  Let $V$ be the natural
module for $H$ (i.e. $2$-dimensional over $F_q$).
Consider $G=V.\mathrm{L}_{2}(q)$ acting on $V$.  Then $G$ is primitive
and $|\Out(G)| = |H^1(H,V)||N_{\mathrm{S}_n}(G):G|= q(q-1)e < (n \log n)/2$
but $ |\Out(G)| > n$ and indeed has order roughly $(n \log n)/2$.
We now show that this is the only example with $|\Out(G)| \geq n$. That will complete the proof of Theorem \ref{main:1}.

\begin{theorem}
\label{out} 
Let $G$ be a primitive affine permutation group
of degree $n$.  Let $V=F(G)=F^*(G)$ and $H \not= 1$ a point stabilizer.
Assume that $H^1(H,V) \ne \{ 0 \}$.  Then either
$|\Out(G)| < n$ or $n=q^2$ with $q=2^e, e > 1$ and
$H=\mathrm{L}_{2}(q)$.
\end{theorem}

\begin{proof}  We use the lemma above and write
$V = \oplus V_i$ where $V_i=[L_i,V]$.  Let  $N$ be the normalizer
of $L_1$.  Then $N$ preserves $V_1$ and indeed $N$ is precisely
the stabilizer of $V_1$.  Since $H$ is irreducible on $V$, $N$
is irreducible on $V_1$.  Let $E={\rm End}_H(V)$.  So $E$ is a field
of size $q$.  By Frobenius reciprocity,
$E \cong {\rm End}_N(V_1)$.
Let $d=\dim_{E} V_1$ and $h_1=|H^1(H,V)|$.

So $H^1(H,V)$ embeds in $H^1(L_1,U_1)$ where $U_1$
is an $L_1$-submodule of $V_1$.  Now we know that
$|\Out(G)| \le (q-1)h_1|N_J(H)/H|$ where $J = \Aut(L) \wr \mathrm{S}_t$.

Suppose that $t=1$.  Then $F^*(H)=L$ and
$|\Out(G)| \le (q-1)h_1|\Out(L)|$.  We first consider
some special cases.

If $d=2$, then (since we are assuming
that $h_1 > 1$), $h_1=q$.  This implies (see \cite{GH})
that $L = \mathrm{L}_{2}(q)$ with $q=2^e > 2$.  Moreover, the equality on
$h_1$ implies that $H=L$ and we are as in the example
above.

Suppose that $d=3$.  Then also by the main result in \cite{GH},
it follows that $h_1=q$.
So $|\Out (G)|<q(q-1)\sqrt{(q^3 -1)/(q-1)}<q^3$ unless possibly $L= \mathrm{A}_6$,
$\mathrm{L}_2(27)$, $\mathrm{L}_3(4)$ or $\mathrm{L}_3(16)$ (see Lemma \ref{l2}). 
In these exceptional cases $|\Out(L)|$ has order $4,6,12$ or $24$ respectively
and the result holds unless possibly $q<|\Out(L)|$. This
easily rules out the first two cases (there are no $3$-dimensional
representations over a field of size $q<|\Out(L)|$ cf.\
\cite{JLPW}).
For the latter two groups, there are no $3$-dimensional
representations (only projective representations).

So assume that $d \ge 4$.

Suppose that $|H^1(H,V)|<|V|^{1/2}$. Since (by construction),
$V$ is an irreducible $\F_q[H]$-module,
$|H^1(H,V)|\le q^{(d/2)-1}$ if $d$ is even and
$|H^1(H,V)|\le q^{(d-1)/2}$ if $d$ is odd.

In this case, the same argument as above applies. We only need to deal with the four groups as above and only for
representations where $|\Out (L)|$ is larger than $\sqrt{(q^d-1)/(q-1)}$. From
\cite{JLPW} we see that the only possibility is that $L= \mathrm{A}_6$
and $q=2$, $d=4$. In that case, $|H^1(H,V)|\leq 2$ and we still
have the result.

The remaining case is when $|H^1(H,V)|=|V|^{1/2}$. This can only
occur for $d=2$ except if $H= \mathrm{A}_6,\ q=3$ and $d=4$ \cite{GH}.
So $n=81$. In that case, we see that $|\Out (G)|\leq 9\cdot 2
\cdot 4=72$  and the result still holds.

Now suppose that $t > 1$ and $n=q^{dt}$ where $d = \dim_E V_1$.
Then we have $|N_J(H)/H| \le |\Out(L)|^{t/2}b_t$ where $b_t$ is the
analogous bound for transitive groups of degree $t$.
Thus by \cite{GH}, $|\Out(G)|\leq (q-1)q^{d/2}|\Out(L)|^{t/2}b_t$.
By Theorem \ref{transitive} and
Theorem \ref{abelian-transitive}, we have $b_t\leq 2^t\cdot t^{\log t}$ and it is easy
to see that if $t\leq 7$ then in fact $b_t\leq 2^t$ holds. Unless
$L= \mathrm{A}_6$ (which we assume) we have $|\Out (L)|< q^{d}/(2(q-1))$ by
\cite[Lemma 2.7]{AG2}.

Assuming first that $q\ne 2$ we see that
$|\Out(G)|<q^{d\left((t/2)+1\right)}t^{\log t}$. It is also
clear that (since $\mathrm{L}_2(3)$ is solvable)
in this case we have $q^d\geq 16$ and $|\Out (G)|<q^{dt}$
follows for $t\geq 8$ (and even for $t \leq 7$ by the observation a few lines above).

Finally let $q=2$. Then we see that $|\Out(G)| \leq 2^{(d+t+dt)/2} \cdot t^{\log t}$. If $d \geq 4$ and $t \geq 8$, then this is less than $2^{dt}$. Otherwise $L = \mathrm{L}_{3}(2)$ and in this case for $t \geq 8$ we have $|\Out(G)| \leq 2 \cdot 2^{t/2} \cdot 2^{t} \cdot t^{\log t} < 2^{3t}$. If $t \leq 7$, then $|\Out(G)| < 2^{dt}$ follows, using $b_{t} \leq 2^{t}$ and Lemma \ref{l2}, in all cases, except when $L = \mathrm{A}_{6} = \mathrm{L}_{2}(9)$. Finally, if $L = \mathrm{A}_{6} = \mathrm{L}_{2}(9)$, then $d \geq 6$ and $|\Out(G)| < 2^{dt}$ follows easily.
\end{proof}

This completes the proof of Theorem \ref{main:1}.

\section{$p$-solvable composition factors and outer automorphism groups}
\label{Section 13}

The purpose of this section is to complete the proof of Theorem \ref{main:3}.

Fix a prime $p$. Throughout this section we put $c$ to be $24^{1/3}$ if $p \leq 3$ and ${(p-1)!}^{1/(p-2)}$ if $p \geq 5$. 

The first result concerns transitive permutation groups.

\begin{theorem}
\label{mixedtransitive}
Let $G \vartriangleleft A \leq \mathrm{S}_{n}$ be transitive permutation groups of degree $n$. Let $p$ be a prime. Then $a_{p}(G) |A/G| \leq c^{n-1}$. 
\end{theorem}

\begin{proof}
We prove the result by induction on $n$. 

First let $A$ be a primitive permutation group. If $A$ contains $\mathrm{A}_n$, then the bound is clear. Otherwise we have $|A| \leq 24^{(n-1)/3}$ by \cite{maroti}. 

Now let $A$ be an imprimitive permutation group. Let $\{ B_{1}, \ldots , B_{t} \}$ be an $A$-invariant partition of the underlying set on which $A$ acts, with $1 < t < n$. Let $A_{1}$ denote the action of the stabilizer of $B_{i}$ in $A$ on $B_{i}$. Then $A$ embeds in $A_{i} \wr \mathrm{S}_{t}$ and $G$ permutes transitively the subgroups $A_{i}$. Let $G_{i}$ denote the action of the stabilizer of $B_{i}$ in $G$ on $B_{i}$. By Theorem \ref{diag1} we have $b(A/G) \leq b(A/GK) b(A_{1}/G_{1})$ where $K$ denotes the kernel of the action of $A$ on $\{ B_{1}, \ldots , B_{t} \}$. 

In order to bound $a_{p}(G)$ we first set $s = |B_{1}|$, that is $n = st$. We have $$a_{p}(G) = a_{p}(G/K \cap G) \cdot a_{p}(K \cap G) \leq a_{p}(GK/K) \cdot a_{p}(K \cap G).$$ Let $N$ be the kernel of the action of $K$ on $B_{1}$. We have $$a_{p}(K \cap G) \cdot a(K/(K \cap G)) \leq a_{p}((K \cap G) N) \cdot a(K / (K \cap G)N) =$$ $$= a_{p}((K \cap G)N/N) \cdot a_{p}(N) \cdot a(K / (K \cap G)N) \leq a_{p}(G_{1}) \cdot a(A_{1}/G_{1}) \cdot a_{p}(N).$$  

We are now in position to bound $a_{p}(G) \cdot |A/G| = a_{p}(G) \cdot a(A/G) \cdot b(A/G)$. We have $$a_{p}(G) \cdot |A/G| \leq a_{p}(GK/K) \cdot |A/GK| \cdot a_{p}(K \cap G) \cdot a(K / (K \cap G)) \cdot b(A_{1}/G_{1}) \leq $$  
$$\leq (a_{p}(GK/K) \cdot |A/GK|) \cdot  (a_{p}(G_{1}) \cdot |A_{1}/G_{1}|) \cdot a_{p}(N).$$
By induction (noting that $GK/K$ and $G_{1}$ are transitive on $\{ B_{1}, \ldots , B_{t} \}$ and $B_{1}$ respectively), we have $a_{p}(GK/K) \cdot |A/GK| \leq c^{t-1}$ and $a_{p}(G_{1}) \cdot |A_{1}/G_{1}| \leq c^{s-1}$. Moreover by repeated use of Proposition \ref{di}, we see that $a_{p}(N) \leq c^{(s-1)(t-1)}$. These give $a_{p}(G) |A/G| \leq c^{t-1} c^{s-1} c^{(s-1)(t-1)} = c^{n-1}$.
\end{proof}

We next need a lemma which depends on the existence of regular orbits under certain coprime actions.

\begin{lemma}
\label{kgv}
Let $A$ be a primitive linear subgroup of $\mathrm{\Gamma L}_{d}(p^{f})$ for a prime $p$ and integers $f$ and $d$ with $d$ as small as possible. Put $A_{0} = \mathrm{GL}_{d}(p^{f}) \cap A$. Let $J_{0}$ be the product of all normal subgroups of $A$ contained in $A_0$ which are nonsolvable, have orders coprime to $p$, and are minimal with respect to being noncentral in $A_{0}$. Then either $p^{f} = 7$, $d = 4$ or $5$, and $b(J_{0}) = |\psp{4}{3}|$, or $b(J_{0}) < p^{fd}$. 
\end{lemma}

\begin{proof}
The group $J_{0}$ is a central product of quasisimple groups $J_{1}, \ldots , J_{r}$ for some $r$. Since $J_{0} \vartriangleleft A$ and $A$ is primitive, $J_{0}$ acts homogeneously on the underlying vector space. Let $W$ be an irreducible $J_{0}$-submodule. As in the proof of Theorem \ref{primitive-linear}, we may use \cite[Lemma 5.5.5, page 205 and Lemma 2.10.1, pages 47-48]{KL} to write $W$ in the form $W_{1} \otimes \cdots \otimes W_{r}$ where for each $i$ the $J_{i}$-module $W_{i}$ (defined over a possibly larger field) is irreducible.  

Assume first that $r = 1$. If $J_{0}$ has a regular orbit on $W$, then the result is clear. If $J_{0}$ does not have a regular orbit on $W$, then \cite[Theorem 7.2.a]{Sc} says that $(J_{0},W)$ is a permutation pair in the sense of \cite[Example 5.1.a]{Sc} or $(J_{0},W)$ is listed in the table on \cite[Page 112]{Sc}. In all these exceptional cases we have $|J_{0}/Z(J_{0})| < |W|$ unless $W$ is a $4$ or $5$ dimensional module over the field of size $7$ and $J_{0}/Z(J_{0}) \cong \psp{4}{3}$. Moreover $|J_{0}/Z(J_{0})| < p^{fd}$ or $|W| = p^{fd}$ and one of the previous exceptional cases holds.

Assume that $r > 1$. For each $i$ we have $b(J_{i}) < |W_{i}|$ or $|W_{i}| = 7^{4}$ or $7^{5}$ and $b(J_{i}) < {|W_{i}|}^{1.31}$. From these it follows that $b(J_{0}) < |W| \leq p^{fd}$.  
\end{proof}

The next lemma may be viewed as a sharper version of Theorem \ref{szar} under the assumption that $p \geq 7$. 

\begin{lemma}
\label{lll}
Let $G$ be a finite group acting faithfully and completely reducibly on a finite vector space $V$ in characteristic $p$. Then 
$a_{p}(G) \leq {|V|}^{c_{1}}/c$.
\end{lemma}

\begin{proof}
By Theorem \ref{szar}, we may assume that $p \geq 5$. By \cite[Theorem 1.1]{HM}, the strong base size of a $p$-solvable finite group $S$ acting completely reducibly and faithfully on a vector space of size $n$ over a field of characteristic $p \geq 5$ is at most $2$. Thus $|S| < n^{2}/(p-1)$. By the proof of Theorem \ref{szar} we then have $a_{p}(G) \leq {|V|}^{2}/(p-1) \leq {|V|}^{c_{1}}/c$.  
\end{proof}

Theorem \ref{mixedtransitive} is used in the proof of the following result which could be compared with Theorem \ref{szar}.
 
\begin{theorem}
\label{mixedirreducible}
Let $G \vartriangleleft A \leq \mathrm{GL}(V)$ be linear groups acting irreducibly on a finite vector space $V$ of size $n$ and characteristic $p$. Then $a_{p}(G) |A/G| \leq 24^{-1/3} n^{c_{1}}$.
\end{theorem}

\begin{proof}
We prove the result by induction on $n$. 

Assume that $A$ acts primitively on $V$. If $a_{p}(G) = a(G)$, then the result follows from part (2) of Lemma \ref{segedlemma}. In fact, in our argument to show Theorem \ref{main:0} we naturally took $G$ to be as small as possible and our calculations actually gave $|A/(G \cap Z(A)J)| < n$ apart from the eleven exceptions listed in Theorem \ref{main:0} (when $a_{p}(G) = a(G)$). Here, as usual, $J$ denotes the central product of all normal subgroups of $A$ contained in $A_0$ (where $A_{0} = \mathrm{GL}_{d}(p^{f}) \cap A$ and $d$ is smallest with $A \leq \mathrm{\Gamma L}_{d}(p^{f})$ and $n = p^{f}$) subject to being noncentral. Thus we are finished if the product of the orders of the nonabelian composition factors of $J$ which are $p'$-groups is at most $24^{-1/3} n^{c_{1}-1}$. Let $J_{0}$ be as in Lemma \ref{kgv}. By Lemma \ref{kgv} we have $b(J_{0}) \leq 24^{-1/3} n^{c_{1}-1}$ unless $n \leq 81$ or $p^{f} = 7$ and $d = 4$ or $5$. It can be checked by GAP \cite{GAP} that the bound holds in case $n \leq 81$. Thus assume that $p^{f} = 7$ and $d = 4$ or $5$ with $|b(J_{0})| = |\psp{4}{3}|$. Then $A$ is a $7'$-group by \cite[Theorem 7.2.a]{Sc} and so $a_{7}(G) |A/G| = |A| \leq 24^{-1/3} n^{c_{1}}$ by \cite[Theorem 1.2]{HM}. 

Assume that $A$ acts (irreducibly and) imprimitively on $V$. Let $V = V_{1} \oplus \cdots \oplus V_{t}$ be a direct sum decomposition of the vector space $V$ such that $1 < t$ and $A$ (and so $G$) acts transitively on the set $\{ V_{1}, \ldots V_{t} \}$. Set $m = |V_{1}|$. Let $K$ be the kernel of the action of $A$ on $\{ V_{1}, \ldots V_{t} \}$ and let $A_1$ and $G_{1}$ be the action of $N_{A}(V_{1})$ and $N_{G}(V_{1})$ on $V_{1}$ respectively. As in the proof of Theorem \ref{mixedtransitive}, we have 
$$a_{p}(G) \cdot |A/G| \leq (a_{p}(GK/K) \cdot |A/GK|) \cdot  (a_{p}(G_{1}) \cdot |A_{1}/G_{1}|) \cdot a_{p}(N),$$
where, in this case, $N$ denotes the kernel of the action of $K$ on $V_{1}$. 
Since the groups $GK/K \vartriangleleft A/K$ can be viewed as transitive permutation groups acting on $t$ points, Theorem \ref{mixedtransitive} gives $a_{p}(GK/K) |A/GK| \leq c^{t-1}$. In the proof of Theorem \ref{linear} it was noted that $G_{1}$ must act irreducibly on $V_{1}$. Thus, by the induction hypothesis, $a_{p}(G_{1}) |A_{1}/G_{1}| \leq 24^{-1/3} m^{c_{1}}$. Since $N$ is subnormal in the irreducible group $A$, it must be completely reducible. By repeated use of Lemma \ref{lll} we have $a_{p}(N) \leq {m}^{c_{1}(t-1)}/c^{t-1}$. Applying these three estimates to the displayed inequality above, we get
$a_{p}(G) \cdot |A/G| \leq c^{t-1} \cdot (24^{-1/3} m^{c_{1}}) \cdot ({m}^{c_{1}(t-1)}/c^{t-1}) = 24^{-1/3} n^{c_{1}}$.
\end{proof}

We are now in the position to complete the proof of Theorem \ref{main:3} in the special case that $G$ is an affine primitive permutation group. 

\begin{theorem}
\label{mixedaffine}
Let $G$ be an affine primitive permutation group of degree $n$, a power of a prime $p$. Then $a_{p}(G) |\Out(G)| \leq 24^{-1/3} n^{1 + c_{1}}$.
\end{theorem}

\begin{proof}
Let $H$ be a point stabilizer in $G$. We may assume that $H \not= 1$. Clearly $H$ acts irreducibly and faithfully on a vector space $V$ of size $n$. If $H^1(H,V) = \{ 0 \}$, then Theorem \ref{mixedirreducible} gives the result, by Lemma \ref{cohomology}. So assume that $H^1(H,V) \not= \{ 0 \}$. By Lemma \ref{AGlemma}, $F^{*}(H) = L_{1} \times \cdots \times L_{t}$ where $L_{i} \cong L$ are nonabelian simple groups viewed as subgroups of $\mathrm{GL}(V_{i})$ where the $V_{i}$ are vector spaces with $V = \oplus_{i=1}^{t} V_{i}$. For each $i$ put $|V_{i}| = p^{d}$ for a prime $p$ and integer $d$. (See Lemma \ref{AGlemma} and the proof of Theorem \ref{out}.)    

If $n=q^2$ with $q=2^e$ for an integer $e > 1$ and $H=\mathrm{L}_{2}(q)$, then, by Section \ref{Section 12}, $a_{2}(G) |\Out(G)| < (n^{2} \log n)/2 < 24^{-1/3} n^{1 + c_{1}}$. Thus, by Theorem \ref{out}, we may assume that $|\Out(G)| < n$. 

Assume first that $L$ is not $p$-solvable. By Lemma \ref{l3} and Proposition \ref{di}, we have the estimate $a_{p}(H) \leq {(4d)}^{t} c^{t-1}$ where $c$ is as in the beginning of this section, depending on $p$. For $d=2$ and $p \geq 5$, $d = 3$ and $p \geq 3$, or $d = 4$ and $p \geq 3$, or $d \geq 5$ we have ${(4d)}^{t} c^{t-1} \leq 24^{-1/3} p^{d t (c_{1} - 1)}$ giving the desired bound for $a_{p}(H)$ and thus for $a_{p}(G) |\Out(G)|$ in these cases. The only exceptions are $d = 3$ and $p = 2$, and $d = 4$ and $p = 2$. However in these cases in the previous two estimates we may replace $4d$, we obtained from Lemma \ref{l3}, by $1$ and $4$ respectively. Thus we may assume that $L$ is $p$-solvable (and $p \geq 3$). 

Let $A$ be the full normalizer of $G$ in $\mathrm{S}_{n}$. Assume that $p \geq 5$ and that $A/G$ is not $p$-solvable. By Schreier's conjecture and the proof of Theorem \ref{transitive} we must then have $t \geq 8$. By the proof of Lemma \ref{lll}
we have $a_{p}(A) \leq n^{3}/4$. Thus, by the main result of \cite{GH}, Lemma \ref{AGlemma} and Theorem \ref{transitive}, we have the estimate $a_{p}(G) |\Out(G)| \leq (n^{3}/4) \cdot t^{\log t} \cdot p^{d/2}$. Furthermore, by the proof of Theorem \ref{out}, we may also assume that $d \geq 3$ as well as $t \geq 8$. Under these conditions it easily follows that $(n^{3}/4) \cdot t^{\log t} \cdot p^{d/2} < 24^{-1/3} n^{1 + c_{1}}$. Assume now that $p \geq 5$ and $A/G$ is $p$-solvable. Then we must bound $a_{p}(A) |H^{1}(H,V)| \leq (n^{3}/4) \cdot p^{d/2}$. Using Lemma \ref{l3} this is smaller than the desired estimate unless $t \leq 2$. Using Lemma \ref{kgv} we can deduce the result if $t \leq 2$ and $d \geq 4$. By the proof of Theorem \ref{out} we cannot have $d=2$ since $H^{1}(H,V) \not= \{ 0 \}$ and $p \not= 2$. Thus $t \leq 2$ and $d=3$. We then have by the proof of Lemma \ref{lll} that $a_{p}(A) |H^{1}(H,V)| \leq (n^{3}/(p-1)) \cdot p$ which is, for $n \geq 343$, less than $24^{-1/3} n^{1+c_{1}}$. This forces $p = 5$, $d=3$ and $t=1$ with $|L|$ not divisible by $5$. GAP \cite{GAP} shows that there is no such possibility.

We are left to consider the case when $p=3$ and $L$ is $3$-solvable, that is, a Suzuki group. In this case $F^{*}(H)$ has a regular orbit on $V$ by \cite[Theorem 7.2.a]{Sc}. We also have $1 < |H^{1}(H,V)| \leq n^{1/14}$ by \cite[Table 2]{H} and so we may assume by Lemma \ref{AGlemma} that $d \geq 14$. By these, Proposition \ref{di} and a remark after Lemma \ref{l2}, $a_{3}(A) |H^{1}(H,V)| < n^{3} < 24^{-1/3} n^{1+c_{1}}$ for $n \geq 3^{14}$. We may thus assume that $a_{3}(G) |\Out(G)| > a_{3}(A) |H^{1}(H,V)|$. Then $t \geq 8$ as in the previous paragraph and so $a_{3}(G) |\Out(G)| < n^{3} \cdot t^{\log t} \cdot p^{d/14} < 24^{-1/3} n^{1+c_{1}}$ whenever $n \geq 3^{8 \cdot 14}$.           
\end{proof}

Finally we finish the proof of Theorem \ref{main:3}. By Theorem \ref{mixedaffine} we may assume that $G$ is a primitive permutation group of degree $n$ with nonabelian socle $E$ which is isomorphic to a direct product of $t$ copies of a nonabelian simple group $L$. By Theorem \ref{main:1} we know that $|\Out(G)| < n$ in this case and so it is sufficient to show $a_{p}(G) \leq 24^{-1/3} n^{c_{1}}$ for every prime divisor $p$ of $n$. This follows from the proof of Corollary \ref{c1}.



\section{On the central chief factors of a primitive group}
\label{Section 13.5}

In this section Theorem \ref{main:3.5} is proved.

For a finite group $X$ let $c(X)$ denote the product of the orders of the central chief factors in a chief series for $X$. By a generalization of the Jordan-H\"older theorem, $c(X)$ is independent from the choice of the chief series for $X$. Clearly, $c(X) \leq c(X/N) c(N)$ for any normal subgroup $N$ of $X$. The following technical lemma is not necessary but it makes our argument easier to understand. 

\begin{lemma}
\label{folosleg}
Let $X = X_{1} \times \cdots \times X_{t}$ be the direct product of finite groups $X_{i}$. Let $Y$ be a subgroup of $X$ such that $Y$ projects onto every factor $X_i$. Then $c(Y)$ is at most $c(X) = \prod_{i=1}^{t} c(X_{i})$. 
\end{lemma}

\begin{proof}
We prove the result by induction on $t$. In case $t = 1$ we have $Y = X$ and the result follows. Assume that $t > 1$. Write $X$ in the form $X_{1} \times X'$ where $X'$ is the direct product of all but the first factor of $X$. Let $N$ be the kernel of the projection of $Y$ onto $X_1$. Then $N \vartriangleleft Y$ and $c(Y)$ is equal to $c(Y/N) = c(X_{1})$ times the product, denoted by $c_{Y}(N)$, of the orders of the central chief factors of $Y$ which are contained in $N$. Let $Y'$ be the projection of $Y$ into $X'$. Now $N$ embeds naturally in $X'$ and $N \vartriangleleft Y'$. Observe that a chief factor of $Y$ contained in $N$ is central in $Y$ if and only if it is central in $Y'$. From this we get $c_{Y}(N) = c_{Y'}(N) \leq c(Y')$. By our induction hypothesis, we get $c(Y') \leq c(X') = \prod_{i=2}^{t} c(X_{i})$.    
\end{proof}

Our first result generalizes a theorem of Dixon \cite{dixon} stating that a nilpotent permutation group of degree $n$ has order at most $2^{n-1}$.

\begin{theorem}
\label{c.permutation}
For a permutation group $G$ of finite degree $n$ we have $c(G) \leq 2^{n-1}$.
\end{theorem} 

\begin{proof}
Let $G$ act on a set $\Omega$ of size $n$. We prove our bound using induction on $n$. 

If $G$ is a cyclic group generated by an $n$-cycle, then $c(G) = n \leq 2^{n-1}$. Assume that $G$ is a noncyclic primitive permutation group. Then the socle of $G$ does not contain any central chief factor of $G$. If $n = 3$, $4$, or $5$, then inspection shows that $c(G)$ is at most $2$, $3$ and $4$ in the respective cases. Otherwise if $n \geq 6$, then Corollary \ref{c1} gives $c(G) \leq 24^{-1/3}n^{c_{1}} < 2^{n-1}$. 

Let $G$ be an imprimitive transitive group acting on $\Omega$. Let $\Sigma$ be a system of nontrivial, proper blocks each of size $k$ and let $K$ be the kernel of the action of $G$ on $\Sigma$. By Lemma \ref{folosleg} and induction, we have $c(K) \leq {(2^{k-1})}^{n/k}$. Thus, again by induction, we get $c(G) \leq c(G/K) c(K) \leq 2^{(n/k) -1} \cdot 2^{n - (n/k)} = 2^{n-1}$.

Finally suppose that $G$ acts intransitively on $\Omega$. Let $\Delta$ be a $G$-orbit in $\Omega$ of size $k$. Let $K$ be the kernel of the action of $G$ on $\Delta$. By induction we have $c(G/K) \leq 2^{k-1}$ and $c(K) \leq 2^{n-k-1}$. Thus $c(G) \leq c(G/K) c(K) < 2^{n-1}$.
\end{proof}

In the next theorem $c_{2}$ denotes the constant $\log_{9}32$ which is close to $1.57732$.

\begin{theorem}
\label{ttt}
Let $G$ be a finite group acting faithfully and completely reducibly on a finite vector space $V$. Then $c(G) \leq |V|^{c_{2}}/2$. 
\end{theorem}

\begin{proof}
Let us prove the result by induction on $n = |V|$. 

Suppose that $G$ acts irreducibly and primitively on $V$. First let $G$ be a subgroup of $\mathrm{\Gamma L}_{1}(p^{f})$ where $n = p^f$ for a prime $p$ and integer $f$. Since $p^{f}-1 < n^{c_{2}}/2$, we may assume that $|G| > p^{f}-1$. Set $G_{0} = G \cap \mathrm{GL}_{1}(p^{f}) < G$. If $G_{0}$ contains a cyclic (irreducible) subgroup of prime order $q$ where $q$ is a Zsigmondy prime, then $c(G) \leq f((p^{f}-1)/q) < p^{f}-1< n^{c_{2}}/2$ since such a cyclic subgroup cannot be central in $G > G_{0}$ and $q \geq f+1$. Otherwise Zsigmondy's theorem implies that $f=2$ or $p^f = 2^6$. In both of these cases we either have $c(G) \leq |G| < n^{c_{2}}/2$ or $G = \mathrm{\Gamma L}_{1}(p^{f})$ and $c(G) = |G|$. If $G = \mathrm{\Gamma L}_{1}(p^{2})$, then $|G| \leq n^{c_{2}}/2$ unless $p = 2$ or $3$ when $c(G) = 2$ and $16$ respectively, with our proposed bound being sharp in the latter case. Thus we may assume that $G = \mathrm{\Gamma L}_{1}(2^{6})$. But in this case $c(G) < |G|$.     

Let us continue to assume that $G$ acts irreducibly and primitively on $V$. Let us use the notation of Theorem \ref{primitive-linear2} with $B$ replaced with $G$. By the previous paragraph we may assume that $d > 1$. By Theorem \ref{primitive-linear2} we may also assume that $n \leq 3^{16}$. Let $G_{0}$ be $B_{0}$ in the notation of the proof of Theorem \ref{primitive-linear2}. Let $J$ be the product of all normal subgroups of $G$ contained in $G_{0}$ which are minimal with respect to not being contained in $C$. 

Assume that $J$ is a central product of quasisimple finite groups. Let $p^{f} = 2$. By the proof of Theorem \ref{primitive-linear2} we may assume in this case that $16 \leq d \leq 25$, and so that $J$ is the direct product of at most two quasisimple groups. Then $c(G)$ is at most $a(G) \leq 32 d < 2^{16} \leq n$. Let $p^{f} \geq 3$. Then $d \leq 16$ and $J$ is the central product of at most three quasisimple groups. In this case $c(G) \leq a(G) \leq 384 f^{4} d (p^{f}-1)$. This is less than $p^{c_{2}fd}/2$ unless $d \leq 7$. Assume that $d \leq 7$. If $J$ is a central product of two quasisimple groups, then, by Dickson's theorem, $a(G) \leq 32 f^{3} (p^{f}-1) < n^{c_{2}}/2$, if $J$ acts on a vector space of dimension $4$, and $a(G) \leq 96 f^{3} (p^{f}-1) < n^{c_{2}}/2$ if $J$ acts on a vector space of dimension $6$. Assume that $J$ is a single quasisimple group. If $d = 2$, then Dickson's theorem gives $a(G) \leq 2 f^{2} (p^{f}-1) < n^{c_{2}}/2$. If $3 \leq d \leq 7$, then $a(G) \leq 4f^{2}d (p^{f}-1)$ which is again smaller than $n^{c_{2}}/2$.         

Assume that $J$ is solvable. Then $c(G) \leq c(G/G_{0}) c(G_{0}/Z(G_{0})) c(Z(G_{0}))$ which is at most $f (p^{f}-1) c(G_{0}/Z(G_{0}))$. The group $G_{0}/Z(G_{0})$ is isomorphic to an affine linear group $HW$ acting on a vector space of size $|W| = d^{2}$ with $W$ a completely reducible $H$-module. By induction it is easy to see that $c(HW) \leq {|W|}^{c_2}/2$. Thus $c(G) \leq f (p^{f}-1) ({d}^{2c_2}/2)$ which is at most $n^{c_{2}}/2$.

Assume now that $J = J_{1} \circ J_{2}$ where $J_{1}$ is a normal subgroup of $G$ that is a central product of quasisimple groups and $J_{2}$ is a solvable noncentral normal subgroup of $G$. Let $d_{1} \geq 2$ and $d_{2} \geq 2$ be the dimensions of the vector spaces (over the field of size $p^f$) on which $J_{1}$ and $J_{2}$ act naturally. So $d_{1}d_{2} \mid d$. The group $G_{0}/Z(G_{0})$ may be viewed in a natural way as a subgroup $Y$ of a certain direct product $X_{1} \times X_{2}$ such that $Y$ projects onto both $X_1$ and $X_{2}$ where $X_{1}$ was treated two paragraphs up and $X_{2}$ one paragraph up as $G_{0}/Z(G_{0})$, where it was shown that $f (p^{f}-1) c(X_{i}) \leq p^{fd_{i}c_{2}}/2$ for $i = 1$ and $2$. By Lemma \ref{folosleg} we have $c(G) \leq f (p^{f}-1) c(X_{1}) c(X_{2}) < p^{fdc_{2}}/2$.    

Assume that $G$ acts irreducibly and imprimitively on $V$. Suppose $G$ preserves a direct sum decomposition $V_{1} \oplus \cdots \oplus V_t$ of the vector space $V$ with $t > 1$ maximal. Let $K$ be the kernel of the transitive permutation action of $G$ on $\{ V_{1}, \ldots , V_{t} \}$. Observe that the action of $K$ on every $V_{i}$ is completely reducible since such a projection of $K$ can be viewed as a normal subgroup in the primitive action of $N_{G}(V_{i})$. Thus by Theorem \ref{c.permutation}, Lemma \ref{folosleg} and induction, we have $$c(G) \leq c(G/K) c(K) \leq 2^{t-1} \cdot {|V_{1}|}^{c_{2}t}/2^{t} = {|V|}^{c_{2}}/2.$$   

Finally assume that $G$ acts reducibly on $V$ which may be considered as $U \oplus W$ where $U$ and $W$ are nontrivial $G$-submodules of $V$. If $K$ denotes the kernel of the action of $G$ on $U$, then $G/K$ and $K$ act faithfully and completely reducibly on $U$ and $W$ respectively. By induction $c(G) \leq c(G/K) c(K) \leq ({|U|}^{c_2}/2) ({|W|}^{c_2}/2) < n^{c_{2}}/2$. This finishes the proof of the theorem.   
\end{proof}

We are now in the position to prove Theorem \ref{main:3.5}. By Theorem \ref{ttt} we are finished in case $G$ is an affine primitive permutation group. If $G$ is not an affine primitive permutation group, then, as in the last two paragraphs of the previous section, $c(G)$ can be bounded using Theorem \ref{c.permutation} and the two remarks that follow Lemma \ref{l2}. In both of these cases we get $c(G) < n^{3/2}/2$. 

\section{Asymptotics}
\label{Section 14}

In this section the second half of Theorem \ref{main:0} and Theorem \ref{main:1.5} are proved. 

Let $G$ be a primitive permutation group of degree $n$ and let $A$ be the full normalizer of $G$ in $\mathrm{S}_n$. Assume that $A$ is not an (affine) subgroup of $\mathrm{A \Gamma L}_{1}(q)$ for a prime power $q$ equal to $n$. We will show that for $n \geq 2^{14000}$ we have $|A/G| < n^{1/2} \log n$.

Assume first that $A$ is an affine primitive permutation group. Then so is $G$. In fact we may change notation and assume that $A$ and $G$ are linear groups acting irreducibly on a finite vector space $V$ of size $n$ with $G \vartriangleleft A$. First assume that $A$ acts primitively on $V$. Let us use the notation of Theorem \ref{primitive-linear3}. By assumption, we have $d \geq 2$. If $d=2$ then, by the structure of $A$, we have that $|A/G|$ is at most $(3/2)(\sqrt{n}-1) \log_{3}n < {n}^{1/2} \log n$ if $A$ is solvable and $|A/G| < {n}^{1/2} \log n$ if $A$ is nonsolvable (using Dickson's theorem). If $d \geq 3$, then, by Theorem \ref{primitive-linear3}, $$|A/G| < n^{1/3} \cdot (\log n) \cdot d^{2 \log d + 3} \leq n^{1/3} {(\log n)}^{2 \log \log n + 4} < n^{1/2}$$ for $n \geq 2^{8192}$. 

Assume now that $A$ acts imprimitively on $V$ and that it preserves a direct sum decomposition $V_{1} \oplus \cdots \oplus V_t$ of $V$ where $t > 1$ is as large as possible. Let $K$ denote the kernel of the (transitive) action of $A$ on $\{ V_{1}, \ldots , V_{t} \}$. As in the proof of Theorem \ref{linear2}, let $A_1$ be the action of $N_{A}(V_{1})$ on $V_{1}$. The group $A_{1}$ acts primitively and irreducibly on $V_{1}$. By Theorem \ref{primitive-linear2}, we have $a(A_{1}) < m \log m$ for $m = |V_{1}| > 3^{16}$. 
In the notation of Lemma \ref{nyuszi} for $t \geq 2^{729}$ we have $a(K/J) \leq (a(X_1))^{t/3c_{1}}$, by use of Lemma \ref{LMM}, where $c_{1}$ is as in Theorem \ref{main:3}. From this and by the proof of Theorem \ref{linear2} together with part (2) of Lemma \ref{segedlemma} and Theorem \ref{main:4}, for $t \geq 2^{729}$ we get $$|A/G| \leq |A/GK| \cdot {(a(A_{1}) b(A_{1}/G_{1}))}^{t/3c_{1}} \leq {16}^{t/\sqrt{\log t}} \cdot n^{1/3} < n^{1/2}.$$ Otherwise, if $t$ is bounded (is at most $2^{729}$) but $t \not= 2$, $4$, then, again by the proof of Theorem \ref{linear2} and by Lemma \ref{nyuszi} and Theorems \ref{main:4} and \ref{linear}, we have, for $n \geq 2^{8192}$, $$|A/G| \leq |A/GK| \cdot b(A_{1}/G_{1}) \cdot {(a(A_{1}))}^{3t/8} \leq {16}^{t/\sqrt{\log t}} \cdot {(\log m)}^{2 \log \log m} \cdot n^{7/16} < n^{1/2}.$$ 

If $t = 2$ and $m \geq 2^{2048}$, then $|A/G| \leq b(A_{1}/G_{1}) a(A_{1}) < m \log m < n^{1/2} \log n$. 

Let $t = 4$. Then $|A/G| \leq 6 \cdot b(A_{1}/G_{1}) \cdot {(a(A_{1}))}^{2}$. Using the notation $d$ (for $A_1$) as in Theorem \ref{primitive-linear3}, by the argument above, we see that for $d \geq 2$ and $m \geq 2^{2048}$ we have $b(A_{1}/G_{1}) \cdot a(A_{1}) < 16 \cdot m^{1/2} \log m$. This gives $|A/G| < n^{1/2} \log n$ for $n \geq 2^{8192}$ under the assumption that $d \geq 2$. Thus assume that $d =1$. Here we use the observation made in Lemma \ref{imprimitiveexample}. Write $a(A_{1}) = |A_{1}|$ in the form $2^{\ell} r$ where $r$ is odd and $\ell$ is an integer. Then we have $|A/G| \leq 6 \cdot 2^{2 \ell} \cdot r$. From this the result follows if $|A_{1}| < 6 m$. Otherwise, by Zsigmondy's theorem, $2^{\ell} < m$. Now $|A_{1}| < m \log_{q} m$ where $q$ is the size of the field over which $A_1$ and $A$ are defined. From this 
$2^{2 \ell} \cdot r < m^{2} \log_{q} m$ giving $|A/G| < n^{1/2} \log n$ unless $q = 2$. If $q = 2$, then $2^{\ell} \leq \log m$, and so $|A/G| < 6 \cdot m {(\log m)}^{2} < n^{1/2} \log n$ for $m \geq 2^{2048}$. 

Assume that $A$ is a primitive permutation group which is not of affine type. In this case we use the notations, assumptions and the argument in the last two paragraphs of Section \ref{Section 9}. However, we use Theorem \ref{main:4}. If $A$ has two minimal normal subgroups, then we have $$|A/G| = a(A/G) b(A/G) < n^{1/3} \cdot {4}^{s/\sqrt{\log s}} \cdot {(\log n)}^{2 \log \log n} < n^{1/2}$$ for $n \geq 2^{8192}$, if $s \not= 1$, and also when $s = 1$. Finally assume that $A$ has a unique minimal normal subgroup. We first claim that we may assume that $t \geq 512$. If $|\mathrm{Out}(L)| \leq \sqrt{m}$ and $t \leq 512$, then $|A/G| \leq n^{1/4} \cdot 16^{t/\sqrt{\log t}} < n^{1/2}$. If $|\mathrm{Out}(L)|$ is larger than $\sqrt{m}$, then $L$ is one of the exceptions in Lemma \ref{l2} and so $t \geq 512$ by use of $n \geq 2^{14000}$. If $t \geq 512$, then, for $n \geq 2^{14000}$, we have $$|A/G| \leq {|\mathrm{Out}(L)|}^{1/3t} \cdot 4^{t/\sqrt{\log t}} \cdot t^{\log t} < n^{1/2} \log n.$$

This proves the second half of Theorem \ref{main:0}. 	
	
So far we showed Theorem \ref{main:1.5} in case $|\Out(G)| = |A/G|$. In fact by the same calculation as in the previous paragraph we established Theorem \ref{main:1.5} in case $G$ is not of affine type. Thus assume that $G$ is of affine type and $H^1(H,V) \ne 0$ where $V$ is the minimal normal subgroup of $G$ and $H \not= 1$ is a point stabilizer in $G$. Let us use the notation and assumptions of the proof of Theorem \ref{out}. 

Let us first assume that $t > 1$ and $|\Out(G)| \leq (q-1)q^{[d/2]}{|\Out(L)|}^{t/2}b_t$ with $d \geq 2$ where $b_{t} \leq {4}^{t/\sqrt{\log t}} \cdot t^{\log t}$, by Theorem \ref{main:4} and $|\Out(L)| \leq 4 (\log n)/t$, by Lemma \ref{l3}. Here $n = q^{dt}$. Thus $|\Out(G)| < q^{[d/2]+1} \cdot {(4(\log n)/t)}^{t/2} \cdot {4}^{t/\sqrt{\log t}} \cdot t^{\log t}$. If $d \geq 3$ or $t \geq 3$, then this is less than $n^{1/2} \log n$ for $n \geq 2^{14000}$. Finally, if $d = 2$ and $t = 2$, then $|\Out(G)| < n^{1/2} \log n$ since $|\Out(L)| \leq \log n$ by use of Dickson's theorem on subgroups $L$ of $\mathrm{GL}_{2}(q)$.    

Finally assume that in the proof of Theorem \ref{out} we have $t=1$, that is, $H$ is an almost simple group with socle $L$. Then $|\Out(G)| \leq (q-1) |H^{1}(L,W)| |\Out(L)|$ where $W$ is a nontrivial irreducible $L$-module of size dividing $n$ and defined over a field of size $q$. By Lemma \ref{l3}, $|\Out(L)| \leq 4 \log n$. 

By the main result of \cite{GH}, we have $|\Out(G)| \leq 4 (q-1) {|W|}^{1/2} \log n$. Assume first that $|W| < n$. If $\dim W \geq 3$, then $|\Out(G)| \leq 4 \cdot n^{5/12} \log n$, which is less than $n^{1/2} \log n$ for $n \geq 2^{14000}$. If $\dim W = 2$, then $|\Out(G)| < n^{1/2} |\Out(L)|$ which is less than $n^{1/2} \log n$ by using Dickson's theorem once again. Thus assume that $|W| = n = q^d$. Furthermore, as observed in the proof of Theorem \ref{out}, we may assume that $d \geq 3$ (otherwise $G$ is a member of the infinite sequence of examples in Theorem \ref{main:1}). Then, by \cite{GH}, it is easy to see that $|\Out(G)| < 2 \cdot n^{3/4}$, at least for $n \geq 2^{14000}$. (In this previous bound the factor $2$ comes from the fact that the full normalizer in $\mathrm{GL}_{4}(q)$ of $\mathrm{A}_6$ acting on the fully deleted permutation module of dimension $4$ in characteristic $3$ over the field of size $q$ has size $(q-1) |\mathrm{S}_6|$ since the dimension of the fixed point space of a $3$-cycle in $\mathrm{A}_6$ is different from the dimension of the fixed point space of an element in $\mathrm{A}_6$ which is the product of two $3$-cycles.) This proves the first statement of Theorem \ref{main:1.5}. 

If $L$ is an alternating group of degree at least $7$, a sporadic simple group or the Tits group, then $\dim(H^{1}(L,W)) \leq (1/4) \dim(W)$ by \cite[Corollary 3]{GK} and \cite{GH}. Thus in these cases we have $|\Out(G)| < 2 \cdot n^{1/2} \leq n^{1/2} \log n$. If $L$ is a simple group of exceptional type, then $\dim(H^{1}(L,W)) \leq (1/3) \dim(W)$ by \cite{H}, thus if $d \geq 7$ then $|\Out(G)| < 4 \cdot n^{3/7} \log n < n^{1/2} \log n$ for $n \geq 2^{14000}$. Otherwise $d = 6$ and $L = \mathrm{G}_{2}(r)$ with $r$ even or $4 \leq d \leq 6$ and $L$ is a Suzuki group (by \cite[Table 5.3.A]{KL} and \cite[Table 5.4.C]{KL}). However in these cases $\dim(H^{1}(L,W)) \leq 1$, by \cite{H}, and so $|\Out(G)| < n^{1/2} \log n$. We may now assume that $L$ is a classical simple group. 

Since we are assuming that the dimension of the natural module for $L$ is at least $7$, we see from \cite[Table 5.4.C]{KL} and \cite[Table 5.3.A]{KL} that $d \geq 7$. By \cite{H} we also have $\dim(H^{1}(L,W)) \leq (1/3) \dim(W)$. Thus $|\Out(G)| < 4 \cdot n^{3/7} \log n < n^{1/2} \log n$ for $n \geq 2^{14000}$.  

This completes the proof of Theorem \ref{main:1.5}.

\end{document}